\documentclass[a4paper,12pt,oneside,reqno]{amsart}
\usepackage[utf8x]{inputenc}
\usepackage[T1]{fontenc}
\usepackage{geometry}
\usepackage{lmodern}

\usepackage{amsmath}
\usepackage{amssymb}
\usepackage{amsthm}
\usepackage[matrix,arrow,cmtip]{xy}
\usepackage{bbm}
\usepackage{enumerate}

\usepackage[hypertexnames=false,bookmarksnumbered]{hyperref}
\hypersetup{colorlinks,linkcolor=black,citecolor=black,urlcolor=black}

\usepackage[capitalise,poorman]{cleveref}

\newcommand{\A}{\mathbb{A}}
\newcommand{\C}{\mathbb{C}}

\newcommand{\N}{\mathbb{N}}

\newcommand{\Q}{\mathbb{Q}}
\newcommand{\R}{\mathbb{R}}

\newcommand{\Z}{\mathbb{Z}}

\newcommand{\cA}{\mathcal{A}}

\newcommand{\Ag}{\mathcal{A}_g}

\newcommand{\fa}{\mathfrak{a}}
\newcommand{\fb}{\mathfrak{b}}

\newcommand{\fo}{\mathfrak{o}}
\newcommand{\fp}{\mathfrak{p}}

\newcommand{\fs}{\mathfrak{s}}

\newcommand{\gG}{\mathbf{G}}
\newcommand{\gH}{\mathbf{H}}
\newcommand{\gS}{\mathbf{S}}
\newcommand{\gU}{\mathbf{U}}

\newcommand{\rM}{\mathrm{M}}

\DeclareMathOperator{\Aut}{Aut}
\DeclareMathOperator{\Br}{Br}
\DeclareMathOperator{\diag}{diag}
\DeclareMathOperator{\End}{End}

\DeclareMathOperator{\GL}{GL}

\DeclareMathOperator{\Hom}{Hom}

\DeclareMathOperator{\Nm}{N}
\DeclareMathOperator{\Nrd}{Nrd}

\DeclareMathOperator{\Tr}{Tr}

\newcommand{\abs}[1]{\left\lvert #1 \right\rvert}

\newcommand{\defterm}[1]{\textbf{#1}}

\newcommand{\emptybar}{\bar{\phantom{v}}}

\newtheorem{lemma}{Lemma}[section]
\newtheorem{proposition}[lemma]{Proposition}
\newtheorem{theorem}[lemma]{Theorem}

\newtheorem{corollary}[lemma]{Corollary}
\newtheorem{conjecture}[lemma]{Conjecture}
\Crefname{conjecture}{Conjecture}{Conjectures}

\newtheorem*{proposition*}{Proposition}
\newtheorem*{theorem*}{Theorem}
\newtheorem*{corollary*}{Corollary}

\newtheorem{nosecclaim}{Claim}
\Crefname{nosecclaim}{Claim}{Claims}

\theoremstyle{definition}

\title[Compatibility between isogenies and polarizations]{On compatibility between isogenies and polarizations of abelian varieties}
\author{Martin Orr}


\newcommand{\cp}[1]{c_{p,#1}}

\numberwithin{equation}{section}

\begin{document}

\begin{abstract}
We discuss the notion of polarized isogenies of abelian varieties, that is, isogenies which are compatible with given principal polarizations.
This is motivated by problems of unlikely intersections in Shimura varieties.
Our aim is to show that certain questions about polarized isogenies can be reduced to questions about unpolarized isogenies or vice versa.

Our main theorem concerns abelian varieties~\( B \) which are isogenous to a fixed abelian variety~\( A \).
It establishes the existence of a polarized isogeny \( A \to B \) whose degree is polynomially bounded in~\( n \), if there exist both an unpolarized isogeny \( A \to B \) of degree~\( n \) and a polarized isogeny \( A \to B \) of unknown degree.
As a further result, we prove that given any two principally polarized abelian varieties related by an unpolarized isogeny, there exists a polarized isogeny between their fourth powers.

The proofs of both theorems involve calculations in the endomorphism algebras of the abelian varieties, using the Albert classification of these endomorphism algebras and the classification of Hermitian forms over division algebras.
\end{abstract}

\maketitle

\section{Introduction}

The goal of this paper is to prove some results about the existence of polarized isogenies between abelian varieties, motivated by work on the André--Pink conjecture on unlikely intersections.
The endomorphism algebra of an abelian variety is a semisimple \( \Q \)-algebra with involution, and polarizations of the abelian variety correspond to positive definite Hermitian forms over this algebra.
Hence most of the paper is in fact concerned with isometries of Hermitian forms over such algebras.

\subsection{Abelian varieties, isogenies and polarizations}

We begin by recalling a number of definitions: isogenies and polarizations of abelian varieties, and the notion of polarized isogenies which are the objects of our main theorems.

An \defterm{abelian variety} is a complete algebraic variety equipped with multiplication and inverse maps which make it into a group object in the category of algebraic varieties over some field.
For the purposes of this paper, it does not matter what the base field is (algebraically closed or not, positive characteristic or characteristic zero).
See \cite{mumford:abelian-varieties} and \cite{milne:abelian-varieties-old} for the main results about abelian varieties.

An \defterm{isogeny} is a homomorphism of abelian varieties which is surjective and finite as a morphism of varieties.
The relation ``there exists an isogeny from \( A \) to \( B \)'' is an equivalence relation on abelian varieties, a bit weaker than isomorphism, which preserves many geometric and arithmetic invariants, for example the endomorphism algebra of the abelian variety.
The \defterm{degree} of an isogeny is its degree as a morphism of varieties.

A \defterm{polarization} of an abelian variety \( A \) is an isogeny \( A \to A^\vee \) (where \( A^\vee \) is the dual abelian variety) which is induced by an ample line bundle on \( A_{\bar{k}} \) according to a certain recipe, whose details are not important here.
We say that a polarization is \defterm{principal} if it has degree~\( 1 \).
Every abelian variety possesses at least one polarization, but not always a principal polarization.
A \defterm{principally polarized abelian variety} is a pair \( (A, \lambda) \) consisting of an abelian variety \( A \) and a principal polarization \( \lambda \) of~\( A \).

The \defterm{endomorphism ring} \( \End A \) of an abelian variety \( A \) is the ring of homomorphisms \( A \to A \).
Note that \( \End A \) may be strictly smaller than the endomorphism ring of \( A_{\bar{k}} \).
The \defterm{endomorphism algebra} of \( A \) is \( \End A \otimes_\Z \Q \).
The endomorphism algebra is a semisimple algebra over \( \Q \) (whatever the base field of~\( A \)) and the endomorphism ring is an order in this algebra.
Any polarization of~\( A \) induces a positive involution, called the \defterm{Rosati involution}, of \( \End A \otimes_\Z \Q \).
If the polarization is principal, then the Rosati involution maps \( \End A \) into itself.

Let \( (A, \lambda) \) and \( (B, \mu) \) be principally polarized abelian varieties.
If \( f \colon A \to B \) is an isogeny, then we obtain a polarization \( f^* \mu \) on \( A \), given by \( f^\vee \circ \mu \circ f \), or equivalently, the polarization induced by the line bundle \( f^* \mathcal{M} \) on \( A \) if \( \mathcal{M} \) is a line bundle on \( B \) inducing \( \mu \).

We then say that \( f \) is a \defterm{polarized isogeny}, or that it is \defterm{compatible with the polarizations}, if
\[ f^* \mu = n.\lambda \text{ for some } n \in \Z. \]
Note that it would be too strict to require in this definition that
\[ f^* \mu = \lambda \]
because \( f^* \mu \) has degree \( (\deg f)^2 \), so this equality can only hold if \( \deg f = 1 \), that is, if \( f \) is an isomorphism.

One motivation for considering polarizations is that we can construct a moduli space \( \Ag \) of principally polarized abelian varieties of dimension \( g \), but not a moduli space of unpolarized abelian varieties.
This moduli space is an example of a Shimura variety.

The significance of polarized isogenies then lies in the fact that two principally polarized abelian varieties are related by a polarized isogeny if and only if the corresponding points of \( \Ag \) lie in the same Hecke orbit.
Hecke orbits are natural equivalence classes on Shimura varieties.

\pagebreak

\subsection{Polarized versus unpolarized isogeny classes}

The relation ``there exists a polarized isogeny from \( (A, \lambda) \) to \( (B, \mu) \)'' is an equivalence relation on polarized abelian varieties which is stronger than the existence of an isogeny from \( A \) to~\( B \) (forgetting the polarizations).
\Cref{isog-inf-hecke} gives an example in which the polarized isogeny class of a principally polarized abelian variety is strictly smaller than its unpolarized isogeny class (namely, an abelian surface with multiplication by a real quadratic field).

Our first main result (section~\ref{sec:fourth-powers}) shows that the existence of an unpolarized isogeny between two principally polarized abelian varieties \emph{does} imply the existence of a polarized isogeny between their fourth powers.
Thus some questions about abelian varieties in an isogeny class can be reduced to questions about abelian varieties in a polarized isogeny class (which may be more natural if one is looking at the moduli space of principally polarized abelian varieties), by replacing the original varieties by their fourth powers.

\begin{theorem} \label{isogeny-classes-become-polarized}
Let \( (A, \lambda) \) and \( (B, \mu) \) be principally polarized abelian varieties over the same base field.
If \( A \) and \( B \) are isogenous, then there is a polarized isogeny from \( (A, \lambda)^4 \) to \( (B, \mu)^4 \).
\end{theorem}

In the proof of \cref{isogeny-classes-become-polarized}, we begin by reducing to the case in which \( A \) is \defterm{isotypic}, that is, isogenous to \( A_0^r \) for some simple abelian variety \( A \).
If \( D \) is the endomorphism algebra of \( A_0 \), then the polarizations \( \lambda \) and \( f^* \mu \) induce positive definite Hermitian forms \( \psi_1 \) and \( \psi_2 \) on \( D^r \).
In order to show that there is a polarized isogeny from \( (A, \lambda)^4 \) to \( (B, \mu)^4 \), we have to show that the direct sum of four copies of \( \psi_1 \) is isometric to a rational multiple of the direct sum of four copies of \( \psi_2 \).
In fact we prove that the previous sentence holds without the words ``a rational multiple of.''

\begin{theorem} \label{fourth-power-hermitian-form}
Let \( (D, *) \) be a division algebra over \( \Q \) with a positive involution.
Let \( V \) be a finite-dimensional right \( D \)-module and let
\[ \psi_1, \psi_2 \colon V \times V \to D \]
be two positive definite \( (D, *) \)-Hermitian forms.

Then \( \psi_1^{\oplus 4} \) and \( \psi_2^{\oplus 4} \) are isometric.
\end{theorem}

We prove \cref{fourth-power-hermitian-form} by breaking it into cases using the Albert classification of division algebras with positive involution, then using the classification of \( (D, *) \)-Hermitian forms from chapter~10 of~\cite{scharlau:quadratic-forms} in each case.
For division algebras of Albert types I, III and IV, isometry of Hermitian forms satisfies a local-global principle so this classification is straightforward.
For division algebras of type II, the isometry class of a Hermitian form is not determined by its localizations, so we must use a result of Lewis~\cite{lewis:skew-hermitian-forms} describing the obstruction to the local-global principle.

\subsection{A degree bound for polarized isogenies}

Our second main theorem (section~\ref{sec:degree-bound}) asserts that, if we fix a principally polarized abelian variety \( (A, \lambda) \) and consider any other principally polarized abelian variety \( (B, \mu) \) in the same polarized isogeny class such that we know the degree \( n \) of an unpolarized isogeny \( A \to B \), then \( A \) and \( B \) are related by a polarized isogeny whose degree is bounded by a polynomial in~\( n \).

\begin{theorem} \label{polarized-isogeny-bound}
Let \( (A, \lambda) \) be a principally polarized abelian variety over a field \( K \).
There exist constants \( c \) and \( k \), depending on \( (A, \lambda) \), such that:

If \( (B, \mu) \) is a principally polarized abelian variety over \( K \) for which
\begin{enumerate}
\item there exists a polarized isogeny \( f \colon A \to B \) (of any degree), and
\item there exists an isogeny \( g \colon A \to B \) of degree \( n \) (not necessarily polarized),
\end{enumerate}
then there exists a polarized isogeny \( h \colon A \to B \) of degree at most \( cn^k \).

The constant \( k \) can be chosen to be \( 4 \dim A \), while the constant \( c \) depends on the endomorphism ring of \( A \).
\end{theorem}

Note that it is obvious, under conditions (1) and (2), that there exists a polarized isogeny \( h \colon A \to B \) whose degree is bounded by some function \( C(A, \lambda, n) \).
This is because there are only finitely many abelian varieties related to \( A \) by an isogeny of degree~\( n \), and each of them has only finitely many principal polarizations (up to polarized isomorphism of principally polarized abelian varieties), by \cite{milne:abelian-varieties-old} Theorem~18.1.
The content of \cref{polarized-isogeny-bound} is that this bound is polynomial in \( n \).

\Cref{polarized-isogeny-bound} is particularly useful in combination with the Masser--Wüstholz isogeny theorem.
If we fix an abelian variety \( A \) over a number field and consider abelian varieties \( B \) over larger number fields such that \( A_{\bar{K}} \) is isogenous to \( B_{\bar{K}} \),
the Masser--Wüstholz theorem asserts that there exists an isogeny \( A_{\bar{K}} \to B_{\bar{K}} \) whose degree is bounded by a polynomial in the degree of the field of definition of \( B \).
The relevant part of the Masser--Wüstholz theorem is as follows.

\begin{theorem}[\cite{mw:isogeny-avs}] \label{mw:isogeny-avs}
Let \( K \) be a number field and \( A \) a principally polarized abelian variety over \( K \).
There exist constants \( c \) and \( \kappa \), depending on \( A \) and \( K \), such that:

If \( B \) is any principally polarized abelian variety over a finite extension \( L \) of \( K \) such that \( A_{\bar{K}} \) is isogenous to \( B_{\bar{K}} \),
then there exists an isogeny \( A_{\bar{K}} \to B_{\bar{K}} \) of degree at most
\[ c(A, K) [L:K]^\kappa. \]
\end{theorem}

The isogeny of bounded degree whose existence is asserted by \cref{mw:isogeny-avs} is not necessarily a polarized isogeny, even if we know initially that \( A_{\bar{K}} \) and \( B_{\bar{K}} \) are in the same polarized isogeny class.
Combining \cref{polarized-isogeny-bound,mw:isogeny-avs} establishes that, in the setting of \cref{mw:isogeny-avs}, if \( (B_{\bar{K}}, \mu) \) is in the \emph{polarized} isogeny class of \( (A_{\bar{K}}, \lambda) \), then there exists a \emph{polarized} isogeny \( A_{\bar{K}} \to B_{\bar{K}} \) satisfying a bound of the same form as \cref{mw:isogeny-avs}.

\subsection{Proof of the degree bound}

The proof of \cref{polarized-isogeny-bound} is easily reduced to a result about endomorphisms of an abelian variety.
We begin by outlining the idea of this proof in the case of an isotypic abelian variety.
As mentioned above, in the isotypic case the polarizations \( \lambda \) and \( g^* \mu \) of \( A \) induce positive definite Hermitian forms \( \psi_1 \) and \( \psi_2 \) on \( D^r \) for a suitable division algebra \( D \) and \( r \in \N \).
The degree~\( n \) of~\( g \) controls the determinants of these forms on a suitable lattice \( \Lambda \) in \( D^r \).
The existence of a polarized isogeny \( f \) implies that \( \psi_2 \) is isometric to a rational scalar multiple of~\( \psi_1 \).
In order to prove the existence of a polarized isogeny \( h \) of bounded degree, we have to show that there is an isometry from \( \psi_2 \) to a rational scalar multiple of \( \psi_1 \) which maps \( \Lambda \) into itself and whose determinant is bounded by a polynomial in \( n \).

However, it seems difficult to reduce the non-isotypic case of \cref{polarized-isogeny-bound} to the isotypic case.
The problem is that if
\[ f_1 \colon (A_1, \lambda_1) \to (B_1, \mu_1),  \;  f_2 \colon (A_2, \lambda_2) \to (B_2, \mu_2) \]
are polarized isogenies, then we only know that
\[ (f_1, f_2)^* (\mu_1, \mu_2) = (n_1 \lambda_1, n_2 \lambda_2) \]
for some integers \( n_1 \) and \( n_2 \), but \( n_1 \) might not be equal to \( n_2 \), and so \( (f_1, f_2) \) might not be a polarized isogeny.

We have found it more convenient to write most of the argument using a symmetric element \( q \in E \) (where \( E \) is a semisimple algebra) instead of Hermitian forms over the simple factors of \( E \).
In the following proposition, the hypothesis that there exists \( a \) such that \( a^\dag q a \in \Q^\times \) corresponds to condition (1) in \cref{polarized-isogeny-bound} while \( \Nm_E(q) \) is related to the degree of \( g \) as appears in condition (2) of \cref{polarized-isogeny-bound}.

\begin{proposition} \label{arith-bound-global}
Let \( (E, \dag) \) be a semisimple \( \Q \)-algebra with involution, let \( R \subset E \) be a \( \dag \)-stable order, and let \( \Nm_E \) be a \( \dag \)-compatible norm on \( E \) of rank~\( d \).

There exists a constant \( c \) depending only on \( (R, \dag, \Nm_E) \) such that:

For every \( q \in R \), if there exists \( a \in E \) such that \( a^\dag q a \in \Q^\times \),
then there exists \( b \in R \) such that
\[ b^\dag q b \in \Z - \{ 0 \} \text{ and } \Nm_E(b) \leq c\Nm_E(q)^{d-1/2}. \]
\end{proposition}

The proof of \cref{arith-bound-global} uses a local-to-global approach.
First we prove that the proposition itself holds for the localization of the algebra \( E \) at each rational prime~\( p \), with a constant~\( c_p \) depending on \( p \).
This part of the proof uses Benoist and Oh's \( p \)-adic polar decomposition~\cite{benoist-oh:polar}.
We then give an independent proof that for all but finitely many~\( p \), we can take \( c_p = 1 \).
In the latter part of the proof, we use Hermitian forms as sketched above for the isotypic case, in particular the integral classification of Hermitian forms over local fields and Shimura's results on maximal lattices.
We then use reduction theory for the adelic points of the multiplicative group of \( E \) to obtain a global result from these local results.

\subsection{Application to the André--Pink conjecture}

The results in this paper on the relationship between isogeny classes and polarized isogeny classes are related to the André--Pink conjecture.
In particular, the results of this paper provide an alternative approach to some parts of the author's recent work on the André--Pink conjecture~\cite{orr:andre-pink}.

The André--Pink conjecture for \( \Ag \) is stated in \cref{andre-pink-ag} below (the definition of ``special subvariety'' is not important for the present discussion).
A key issue in studying this conjecture is the relationship between \cref{andre-pink-ag} and the \textit{a priori} slightly stronger \cref{andre-pink-ag-isog}.
Viewed from the perspective of Shimura varieties, the natural statement is \cref{andre-pink-ag} -- a generalization to arbitrary mixed Shimura varieties can be found at \cite{pink:conj} Conjecture~1.6. 
On the other hand, viewed from the perspective of abelian varieties, \cref{andre-pink-ag-isog} appears more natural.

\begin{conjecture}[André--Pink] \label{andre-pink-ag}
Let \( Z \) be an irreducible algebraic subvariety of the moduli space \( \Ag \) of principally polarized abelian varieties of dimension~\( g \) over~\( \C \).

If there exists a polarized isogeny class \( \Sigma \) in \( \Ag \)
such that \( \Sigma \cap Z \) is Zariski dense in~\( Z \), then \( Z \) is a special subvariety of \( \Ag \).
\end{conjecture}

\begin{conjecture} \label{andre-pink-ag-isog}
\Cref{andre-pink-ag} holds with ``a polarized isogeny class \( \Sigma \)'' replaced by ``an isogeny class \( \Sigma \).''
\end{conjecture}

In \cite{orr:andre-pink}, the author proved some cases of \cref{andre-pink-ag-isog}.
In particular, this includes the case when \( Z \) is a curve, and some partial progress on other cases.

\begin{theorem} (\cite{orr:andre-pink} Theorem~1.2) \label{andre-pink-curve-in-ag}
\cref{andre-pink-ag-isog} holds when \( Z \) is a curve.
\end{theorem}

The proof of \cref{andre-pink-curve-in-ag} relies on the Masser--Wüstholz isogeny theorem and the Pila--Zannier method for solving unlikely intersections problems in Shimura varieties.
The proof is complicated by the fact that a straightforward application of the Pila--Zannier method would only apply to polarized isogenies, while the Masser--Wüstholz theorem concerns unpolarized isogenies.
In \cite{orr:andre-pink}, this is worked around by a more sophisticated application of the Pila--Zannier method, as discussed in \cite{orr:andre-pink} section~3.2.

The results in this paper explore the relationship between polarized and unpolarized isogenies directly instead of bypassing it, and thereby give an alternative proof of \cref{andre-pink-curve-in-ag}.
We can use \Cref{polarized-isogeny-bound} to replace most of the difficult parts of \cite{orr:andre-pink}, leading to a proof of \cref{andre-pink-ag} for a curve~\( Z \).
\Cref{isogeny-classes-become-polarized} allows us to deduce \cref{andre-pink-ag-isog} for a curve in \( \Ag \) from~\cref{andre-pink-ag} for a curve in \( \cA_{4g} \).

Nevertheless the proofs of the theorems in this paper are sufficiently complicated that the proof of \cref{andre-pink-curve-in-ag} sketched above seems to be longer overall than the proof in \cite{orr:andre-pink}.

\section{Background: Hermitian Forms and Involutions}

Before the main proofs of this paper, we collect some basic facts about Hermitian forms over division algebras.
We begin with some general definitions and facts, to fix the terminology and notation we will use, which is based on~\cite{knus+:involutions}.
We then state results about two particular aspects of the theory of Hermitian forms: positive definite forms, using~\cite{kottwitz:positive-involutions}, and lattices, using~\cite{shimura:alternating-forms} and~\cite{shimura:hermitian-forms-book}.

\subsection{General Hermitian forms and involutions}

Let \( (D, *) \) be a simple algebra with involution.
By an \defterm{involution}, we mean an additive map \( D \to D \) whose square is the identity and which \emph{reverses} the direction of multiplication.
When we say \defterm{simple algebra with involution}, we include the case in which \( D \) is a product of two simple algebras \( D_1 \times D_2 \) and the involution exchanges the two factors (this is not simple as an algebra, but it is simple as an algebra-with-involution).
The involution~\( * \) is said to be \defterm{of the first kind} if it is trivial on the centre of \( D \), and \defterm{of the second kind} otherwise.

We say that an element \( d \in D \) is \defterm{symmetric} if \( d^* = d \).

Let \( V \) be a free right \( D \)-module of finite dimension.
A \defterm{\( (D, *) \)-Hermitian form} on \( V \) is a bi-additive map
\[ \psi \colon V \times V \to D \]
satisfying
\begin{enumerate}
\item \( \psi(va, wb) = a^* \psi(v, w) b \)
for all \( v, w \in V \) and \( a, b \in D \), and
\item \( \psi(w, v) = \psi(v, w)^* \)
for all \( v, w \in V \).
\end{enumerate}
A \defterm{\( (D, *) \)-skew-Hermitian form} on \( D^n \) is a bi-additive map satisfying condition~(1) above and also
\( \psi(w, v) = -\psi(v, w)^* \).

A Hermitian or skew-Hermitian form is \defterm{non-singular} (also called \defterm{regular}) if the only element \( v \in V \) such that \( \psi(v, w) = 0 \) for all \( w \in V \) is \( v = 0 \).

Given any non-singular Hermitian or skew-Hermitian form \( \psi \colon V \times V \to D \), there is a unique involution \( \dag \) of \( \End_D(V) \), called the \defterm{adjoint involution} with respect to \( \psi \), such that
\[ \psi(av, w) = \psi(v, a^\dag w) \]
for all \( v, w \in V \) and \( a \in \End_D(V) \).

The following proposition shows that we can reverse the construction of adjoint involutions, passing from an involution to a Hermitian or skew-Hermitian form.
This proposition is \cite{knus+:involutions} Proposition~I.4.2 whenever \( D \) is simple as an algebra (forgetting the involution), and it follows from \cite{knus+:involutions} Proposition~I.2.14 whenever \( D \) is a product of two simple algebras.

\begin{proposition} \label{adjoint-involutions}
Let \( (D, *) \) be a simple algebra with involution.
Let \( V \) be a free right \( D \)-module of finite dimension.

\begin{enumerate}
\item If \( * \) is of the first kind, then the map sending a form to its associated adjoint involution induces a bijection between
\begin{enumerate}
\item non-singular \( (D, *) \)-Hermitian and \( (D, *) \)-skew-Hermitian forms on~\( V \), modulo multiplication by an element of the centre of \( D \), and
\item involutions of \( \End_D(V) \) of the first kind.
\end{enumerate}

\item If \( * \) is of the second kind, then the map sending a form to its associated adjoint involution induces a bijection between
\begin{enumerate}
\item non-singular \( (D, *) \)-Hermitian forms on \( V \), modulo multiplication by an element of the centre of \( D \) which is fixed by \( * \), and
\item involutions \( \dag \) of \( \End_D(V) \) such that \( a^\dag = a^* \) for all \( a \) in the centre of~\( D \).
\end{enumerate}
\end{enumerate}
\end{proposition}

The natural notion of equivalence between Hermitian or skew-Hermitian forms is isometry.
We say that two right \( D \)-modules \( V_1 \) and \( V_2 \) equipped with Hermitian or skew-Hermitian forms \( \psi_1 \) and \( \psi_2 \) respectively are \defterm{isometric} if there is an isomorphism of \( D \)-modules \( f \colon V_1 \to V_2 \) such that
\[ \psi_2(f(v), f(w)) = \psi_1(v, w) \text{ for all } v, w \in V_1. \]

\subsection{Positive definite forms and involutions}

In this section we study positivity properties of Hermitian forms on semisimple \( \Q \)- and \( \R \)-algebras with involution.

We begin by making definitions when \( (D, *) \) is a semisimple \( \R \)-algebra with involution.
Let \( V \) be a finite-dimensional right \( D \)-module.

We say that a \( (D, *) \)-Hermitian form \( \psi \colon V \times V \to D \) is \defterm{positive definite} if
\[ \Tr(\psi(v, v) \, ; \, D) > 0 \text{ for all } v \in V - \{ 0 \}, \]
where the trace is taken with respect to the action of \( D \) on itself (viewed as an \( \R \)-vector space) by left multiplication.
Note that a skew-Hermitian form can never be positive definite because it will have \( \Tr (\psi(v, v) \, ; \, D) = 0 \) for all \( v \in V \).

We say that a symmetric element \( d \in D \) is \defterm{positive} if the \( (D, *) \)-Hermitian form on \( D \) itself given by
\[ (v, w) \mapsto v^* d w \]
is positive definite.

We say that \( * \) is a \defterm{positive involution} if \( 1 \) is a positive element of \( D \) with respect to \( * \).
In other words, the bilinear form \( (v, w) \mapsto \Tr(v^* w \, ; \, D) \colon D \times D \to \R \) is positive definite.

If \( (D, *) \) is a semisimple \( \Q \)-algebra with involution, then we make the same definitions as above with \( \Tr_{D/\R} \) replaced by \( \Tr_{D/\Q} \).
In other words, an involution of a \( \Q \)-algebra~\( D \) is positive if and only if its extension to \( D \otimes_\Q \R \) is positive, and similarly for the other definitions.

The lemmas below apply to both semisimple \( \R \)-algebras and semisimple \( \Q \)-algebras \( D \).
For each lemma we either begin by reducing to the case of \( \R \)-algebras, or the proof applies directly to both cases.

\begin{lemma} \label{positive-adjoint}
If \( \psi \colon V \times V \to D \) is a positive definite \( (D, *) \)-Hermitian form, then the associated adjoint involution \( \dag \colon \End_D(V) \to \End_D(V) \) is positive.
\end{lemma}

\begin{proof}
If the base field is \( \Q \), replace \( D \) by \( D \otimes_\Q \R \).
This does not change either the premise or the conclusion of the lemma.

By \cite{kottwitz:positive-involutions} Lemma~2.2, \( \dag \) is a positive involution if and only if
\begin{equation} \label{eqn:trace-endrv}
\Tr (xx^\dag \, ; \, \End_\R(V)) > 0
\end{equation}
for all \( x \in \End_D(V) - \{ 0 \} \),
where \( \End_D(V) \) acts on \( \End_\R(V) \) by left multiplication.

Let \( \theta \) denote the adjoint involution of \( \End_\R(V) \) with respect to the symmetric \( \R \)-bilinear form
\[ \Tr(\psi(-, -) \, ; \, D) \colon V \times V \to \R. \]
Observe that \( \dag \) is the restriction of \( \theta \) to \( \End_D(V) \), so in order to prove \eqref{eqn:trace-endrv} it suffices to prove that \( \theta \) is a positive involution of \( \End_\R(V) \).

All positive definite symmetric bilinear forms on a real vector space are isometric, so we can replace \( \Tr \psi \) by the standard symmetric form on \( \R^n \) (where \( n = \dim_\R V \)).
This replaces \( \theta \) by the transpose involution of \( \rM_n(\R) \), which is well-known to be a positive involution.
\end{proof}

\begin{lemma} \label{positive-definite-form-multiply}
If \( \psi \colon V \times V \to D \) is a positive definite \( (D, *) \)-Hermitian form and \( q \in \End_D(V) \) is symmetric and positive with respect to the adjoint involution \( \dag \) associated with \( \psi \), then
\[ \psi_q \colon (v, w) \mapsto \psi(v, qw) \]
is a positive definite \( (D, *) \)-Hermitian form on \( V \).
\end{lemma}

\begin{proof}
If the base field is \( \Q \), replace \( D \) by \( D \otimes_\Q \R \).
This does not change either the premise or the conclusion of the lemma.

The fact that \( q \) is symmetric implies that \( \psi \) is Hermitian.

By \cref{positive-adjoint}, \( \dag \) is a positive involution of \( \End_D(V) \).
Hence we can apply \cite{kottwitz:positive-involutions} Lemma~2.8 to obtain \( b \in \End_D(V) \) such that \( q = b b^\dag \).

We then have
\[ \psi_q(v, v) = \psi(v, b b^\dag v) = \psi(b^\dag v, b^\dag v) > 0 \]
for all \( v \in V - \{ 0 \} \).
\end{proof}

\begin{lemma} \label{positive-involutions-restrict}
Suppose that \( D \) is a division algebra.
Let \( E = \End_D(V) \) and suppose that we are given a positive involution \( \dag \) of \( E \).

Then there exist
\begin{enumerate}
\item a positive involution \( * \) of \( D \), and
\item a positive definite \( (D, *) \)-Hermitian form \( \psi \colon V \times V \to D \)
\end{enumerate}
such that \( \dag \) is the adjoint involution associated with \( \psi \).
\end{lemma}

\begin{proof}
This proof applies directly to both \( \Q \)-algebras and \( \R \)-algebras.

The algebras \( D \) and \( E \) are similar.
Hence \cite{knus+:involutions} Proposition~I.3.1 tells us that \( D \) possesses an involution \( ! \) whose restriction to the centre is the same as that of \( \dag \).
By \cref{adjoint-involutions}, there exists a \( (D, !) \)-Hermitian or -skew-Hermitian form \( \phi \colon V \times V \to D \) such that \( \dag \) is the adjoint involution with respect to \( \phi \).
However, the involution~\( ! \) might not be positive and \( \phi \) might not be a positive definite Hermitian form.

By \cref{phi-nondeg} below, there is some \( v_0 \in V \) such that \( \phi(v_0, v_0) \neq 0 \).
Let
\[ s = \phi(v_0, v_0) \in D^\times \]
and observe that \( s^! = \epsilon s \), where \( \epsilon = +1 \) or \( -1 \) according as \( \phi \) is \( (D, !) \)-Hermitian or -skew-Hermitian.

Define a new involution \( * \) of \( D \) and a new bi-additive map \( \psi \colon V \times V \to D \) by
\[ d^* = s^{-1} d^! s \]
and
\[ \psi(v, w) = s^{-1} \phi(v, w). \]
Calculations show that \( \psi \) is a \( (D, *) \)-Hermitian form (regardless of the sign of~\( \epsilon \)) and that \( \dag \) is the associated adjoint involution of \( E \).
The facts that \( * \) is positive and that \( \psi \) is positive definite are \cref{psi-pos-def,star-positive} below.

\begin{nosecclaim} \label{phi-nondeg}
There exists \( v_0 \in V \) such that \( \phi(v_0, v_0) \neq 0 \).
\end{nosecclaim}

Assume for contradiction that \( \phi(v, v) = 0 \) for all \( v \in V \).
Since \( \phi \) is non-singular, we can choose \( v_1, v_2 \in V \) such that \( \phi(v_1, v_2) \neq 0 \).
By multiplying \( v_2 \) by a suitable element of \( D \), we can assume that \( \phi(v_1, v_2) = 1 \).

Define a \( D \)-endomorphism \( e \colon V \to V \) by
\[ e(v) = v_1 \, \epsilon \phi(v_2, v) - v_2 \, \phi(v_1, v). \]
A calculation shows that \( e^\dag = -e \).
Further calculations (using the assumption that \( \phi(v_1, v_1) = \phi(v_2, v_2) = 0 \)) give
\[ ee^\dag(v_1) = -v_1, \quad ee^\dag(v_2) = -v_2, \quad ee^\dag(v) = 0 \text{ if } \phi(v_1, v) = \phi(v_2, v) = 0. \]
It follows that \( ee^\dag \) acts as multiplication by \( -1 \) on the right \( D \)-module spanned by \( v_1 \) and~\( v_2 \), and as multiplication by \( 0 \) on the right \( D \)-module 
\[ \{ v \in V : \psi_0(v_1, v) = \psi_0(v_2, v) = 0 \}. \]
These two submodules span~\( V \), and so \( \Tr(ee^\dag \, ; \, V) < 0 \).
According to \cite{kottwitz:positive-involutions} Lemma~2.2, this contradicts the positivity of \( \dag \).

\begin{nosecclaim} \label{psi-pos-def}
\( \psi \colon V \times V \to D \) is positive definite.
\end{nosecclaim}

For each \( v \in V - \{ 0 \} \), define an endomorphism \( e_v \in E \) by
\[ e_v(w) = v \, \psi(v_0, w). \]
By construction, \( \psi(v_0, v_0) = s^{-1} s = 1 \) and so
\[ e_v(v_0) = v. \]

A calculation shows that
\[ e_v^\dag(w) = v_0 \, \psi(v, w) \text{ for all } w \in V \]
and hence
\[ e_v^\dag e_v(v_0) = v_0 \, \psi(v, v). \]
Meanwhile, \( e_v^\dag e_v(w) = 0 \) if \( \psi(v_0, w) = 0 \).
The submodules \( v_0 D \) and
\[ \{ w \in V : \psi(v_0, w) = 0 \} \]
span \( V \) and so
\[ \Tr(e_v^\dag e_v \, ; \, V) = \Tr(\psi(v, v) \, ; \, D). \]

Because \( \dag \) is positive, \cite{kottwitz:positive-involutions} Lemma~2.2 implies that \( \Tr(e_v^\dag e_v \, ; \, V ) > 0 \).
Hence we have shown that \( \psi \) is positive definite.

\begin{nosecclaim} \label{star-positive}
\( * \) is a positive involution of \( D \).
\end{nosecclaim}

For each \( d \in D - \{ 0 \} \), we have
\[ d^* d = d^* \, \psi(v_0, v_0) \, d = \psi(v_0 d, v_0 d). \]
Hence the fact that \( \psi \) is positive definite implies that \( * \) is a positive involution. 
\end{proof}

\subsection{Lattices and Hermitian forms} \label{ssec:lattices}

Let \( S_0 \) be a Dedekind domain and \( F_0 \) its field of fractions.
Let \( F \) be one of the following \( F_0 \)-algebras:
\begin{enumerate}[(i)]
\item \( F = F_0 \);
\item \( F \) is a separable quadratic extension of \( F_0 \);
\item \( F = F_0 \times F_0 \).
\end{enumerate}
Let \( S \) be the integral closure of \( S_0 \) in \( F \).

Let \( x \mapsto \bar{x} \) denote the identity automorphism of \( F \) in case (i), and the non-trivial element of \( \Aut(F/F_0) \) in cases (ii) and (iii).

Throughout section~\ref{ssec:lattices}, we assume that:
\begin{enumerate}[(a)]
\item \( 2 \) is invertible in \( S_0 \); and
\item in case (ii), \( F/F_0 \) is unramified at all primes of \( S_0 \).
\end{enumerate}

Let \( V \) be a finite-dimensional \( F \)-module and \( \psi \colon V \times V \to F \) an \( F_0 \)-bilinear form of one of the following types:
\begin{enumerate}[(i)]
\item if \( F = F_0 \), then \( \psi \) is either symmetric or skew-symmetric;
\item otherwise, \( \psi \) is \( (F, \emptybar) \)-Hermitian.
\end{enumerate}

By a \defterm{lattice} in a finite-dimensional \( F \)-module \( V \), we mean a finitely generated \( S \)-submodule which spans \( V \) over \( F \).

The \defterm{scale} \( \fs\Lambda \) of a lattice \( \Lambda \) (with respect to \( \psi \)) is the fractional ideal of \( F \) generated by \( \psi(\Lambda, \Lambda) \).
In paragraph~4.6 of \cite{shimura:hermitian-forms-book}, this is denoted \( \mu_0(\Lambda) \).

If \( F = F_0 \) and \( \psi \) is symmetric or if  \( F \neq F_0 \), then one can also define the \defterm{norm ideal} \( \mu(\Lambda) \) to be the fractional ideal of \( F_0 \) generated by \( \{ \psi(v, v) : v \in \Lambda \} \).
Our hypotheses that \( 2 \) is invertible in \( S \) and that \( F/F_0 \) is unramified imply that
\[ \fs\Lambda = \mu(\Lambda) S_0. \]
It follows that \( \fs\Lambda \) and \( \mu(\Lambda) \) determine each other, and we are free to use \( \fs\Lambda \) in place of \( \mu(\Lambda) \) when applying results from \cite{shimura:hermitian-forms-book}.

If \( F = F_0 \) and \( \psi \) is skew-symmetric, then the ideal called \( \mu(\Lambda) \) in the above paragraph is equal to zero.
Hence in this case it only makes sense to consider \( \fs\Lambda \), and it is the same as what is called \( N(\Lambda) \) in \cite{shimura:alternating-forms}.

We say that a lattice \( \Lambda \) is \defterm{maximal} with respect to \( \psi \) if there is no lattice which strictly contains \( \Lambda \) and which has the same scale as \( \Lambda \).
We say that \( \Lambda \) is \defterm{\( \fa \)-maximal} if \( \Lambda \) is maximal and \( \fs\Lambda = \fa \).

In section~\ref{sec:degree-bound}, we will use the following facts about maximal lattices.
In the following lemmas, we assume that \( F \), \( F_0 \), \( S \), \( S_0 \), \( V \) and \( \psi \colon V \times V \to F \) are as above.

\begin{lemma} \label{lattice-in-maximal}
Let \( \Lambda \) be a lattice in \( V \) and let \( \fa \) be a fractional ideal of \( F_0 \) such that
\( \fs\Lambda \subset \fa \).

Then there exists an \( \fa \)-maximal lattice in \( V \) which contains \( \Lambda \).
\end{lemma}

\begin{proof}
If \( F = F_0 \) and \( \psi \) is symmetric or if \( F \neq F_0 \), then this is \cite{shimura:hermitian-forms-book} Lemma~4.8.

If \( F = F_0 \) and \( \psi \) is skew-symmetric, then this is the sentence immediately preceding Proposition~1.4 in \cite{shimura:alternating-forms}.
\end{proof}

\begin{lemma} \label{maximal-lattices-are-isometric}
Let \( \fa \) be a fractional ideal of \( F_0 \).
All \( \fa \)-maximal lattices in \( V \) are isometric.
\end{lemma}

\begin{proof}
If \( F = F_0 \) and \( \psi \) is symmetric or if \( F/F_0 \) is a quadratic field extension, then this is \cite{shimura:hermitian-forms-book} Lemma~5.9.

If \( F = F_0 \times F_0 \), then this is \cite{shimura:hermitian-forms-book} Lemma~4.12.

If \( F = F_0 \) and \( \psi \) is skew-symmetric, then this follows from \cite{shimura:alternating-forms} Proposition~1.4.
\end{proof}

\begin{lemma} \label{symmetric-stabilizer-implies-maximal-lattice}
Let \( \dag \) be the adjoint involution of \( \End_F(V) \) with respect to \( \psi \).

Let \( \Lambda \subset V \) be a lattice and let \( R \) be the stabilizer in \( \End_F(V) \) of \( \Lambda \).
Suppose that \( R \) is a maximal order in \( \End_F(V) \) and that \( \dag \) maps \( R \) into itself.

Then \( \Lambda \) is a maximal lattice with respect to \( \psi \).
\end{lemma}

\begin{proof}
Let \( \fa = \fs\Lambda \) and let
\[ \Lambda^* = \{ v \in V \mid \psi(v, \Lambda) \subset \fa \}. \]
Observe that any lattice which contains \( \Lambda \) and which has scale \( \fa \) must be contained in \( \Lambda^* \).
Hence in order to prove that \( \Lambda \) is maximal, it suffices to show that \( \Lambda = \Lambda^* \).

Since \( R \) is \( \dag \)-stable, we have
\[ \psi(Rv, w) = \psi(v, Rw) \subset \fa \text{ for all } v \in \Lambda^*, w \in \Lambda. \]
Hence \( R \) stabilizes \( \Lambda^* \).
Since \( R \) is a maximal order in \( \End_F(V) \), it follows that \( R \) is equal to the stabilizer of \( \Lambda^* \).

In other words \( \Lambda \) and \( \Lambda^* \) have the same stabilizer, and so \( \Lambda^* = u\Lambda \) for some scalar \( u \in F^\times \).
This implies that
\[ \psi(\Lambda^*, \Lambda) = \bar{u} \psi(\Lambda, \Lambda). \]
But the definition of \( \Lambda^* \) implies that \( \psi(\Lambda^*, \Lambda) \subset \fa \) so \( \bar{u} \fa \subset \fa \).
This implies that \( \bar{u} \in S \) so also \( u \in S \).
Hence \( \Lambda^* \subset \Lambda \).

The inclusion \( \Lambda \subset \Lambda^* \) is obvious.
\end{proof}

\section{An Example of Polarized Isogeny Classes} \label{sec:example}

We give an example to show that polarized isogeny classes truly can be smaller than isogeny classes.

The monoid of polarizations of an abelian variety depends on its endomorphism ring, and hence the same is true for the number of polarized isogeny classes contained in the isogeny class of that abelian variety.
If \( (A, \lambda) \) is a principally polarized abelian variety such that \( \End A \) is either \( \Z \) or an order in an imaginary quadratic field, then all polarizations of \( A \) must have the form \( n.\lambda \) for some \( n \in \Z \).
Hence in these cases (which include all elliptic curves over fields of characteristic zero), all isogenies from \( A \) to another principally polarized abelian variety are automatically polarized isogenies.

The following proposition shows that this fails when we consider the next simplest case, namely abelian surfaces with multiplication by a real quadratic field.

\begin{proposition} \label{isog-inf-hecke}
Let \( (A, \lambda) \) be a principally polarized abelian variety over an algebraically closed field, such that \( \End(A) \) is the ring of integers of a real quadratic field.

There are infinitely many distinct polarized isogeny classes of principally polarized abelian varieties, all isogenous to \( A \).
\end{proposition}

\begin{proof}
Let \( \fo_F = \End A \) and let \( F = \fo_F \otimes_\Z \Q \).

For each totally positive element \( q \in \fo_F \), there is a principally polarized abelian surface \( (A_q, \lambda_q) \) and an isogeny \( f_q \colon A \to A_q \) such that
\[ f_q^* \lambda_q = \lambda \circ q. \]
This can be seen by applying the proof of \cite{mumford:abelian-varieties} p.~234 Corollary~1 to \( A \), using a line bundle \( L \) associated with the polarisation \( \lambda \circ q \).

Suppose that for two totally positive elements \( q \) and \( r \in \fo_F \), there exists a polarized isogeny \( g \colon A_q \to A_r \).
By definition, we have
\[ g^* \lambda_r = n\lambda_q \]
for some \( n \in \Z \).
Letting \( u = f_r^{-1} g f_q \in \End A \otimes_{\Z} \Q \), we find that
\[ nq = u^\dag ru \]
where \( \dag \) is the Rosati involution of \( \End A \) induced by \( \lambda \).
In the case we are considering, where the endomorphism algebra is a real quadratic field, the Rosati involution is the identity.

We conclude that \( (A_q, \lambda_q) \) and \( (A_r, \lambda_r) \) are in the same polarized isogeny class if and only if there exist \( n \in \Z \) and \( u \in \fo_F \) such that
\[ nq = u^2 r, \]
or equivalently if and only if
\[ q/r \in \Q^\times F^{\times 2}. \]

It follows that \( q \mapsto (A_q, \lambda_q) \) is an injection from \( F^{+,\times} / \Q^{+,\times} F^{\times 2} \) to the set of polarized isogeny classes of principally polarized abelian varieties isogenous to~\( A \), where \( F^{+,\times} \) means the multiplicative group of totally positive elements of \( F \).

The group \( F^{+,\times} / \Q^{+,\times} F^{\times 2} \) is infinite because there are infinitely many rational primes which split in \( \fo_F \) as a product of two principal prime ideals
\[ (p) = (a_p)(a_p'). \]
In this splitting, we can always choose \( a_p \) totally positive, and the elements \( a_p \in \fo_F \) for different \( p \) are in different classes in \( F^{+,\times} / \Q^{+,\times} F^{\times 2} \).
\end{proof}

\section{Polarized Isogenies and Fourth Powers of Abelian Varieties} \label{sec:fourth-powers}

In this section we will prove \cref{isogeny-classes-become-polarized}, that is, if two principally polarized abelian varieties are isogenous then their fourth powers are in the same polarized isogeny class.
The proof uses \cref{fourth-power-hermitian-form}, which asserts that for all positive definite Hermitian forms over a division \( \Q \)-algebra with positive involution, the isometry class of the direct sum of four copies of the Hermitian form does not depend on the form we started with.

\subsection{Proof that \cref{fourth-power-hermitian-form} implies \cref{isogeny-classes-become-polarized}}

Let \( (A, \lambda) \) and \( (B, \mu) \) be principally polarized abelian varieties and let \( f \colon A \to B \) be an isogeny.
Let \( E = \End A \otimes_\Z \Q \).
Then \( E \) is a semisimple \( \Q \)-algebra equipped with a positive involution~\( \dag \), the Rosati involution with respect to the polarization~\( \lambda \).

Now \( f^* \mu \) is a polarization of \( A \) and so there is a symmetric endomorphism \( q \in E \) such that
\begin{equation} \label{eqn:def-q}
f^* \mu = \lambda \circ q.
\end{equation}

Furthermore, \( q \) is positive with respect to \( \dag \).
The positivity can be proved by adapting the proof of \cite{mumford:abelian-varieties} \S 21 Theorem~1:
if \( \lambda \) and \( \lambda \circ q \) are the polarizations associated with the divisors \( D \) and \( D_q \) respectively, and \( E^\lambda \) and \( E^{\lambda q} \) are the associated Riemann forms, then \cite{mumford:abelian-varieties} \S 20 Theorem~3 tells us that for any \( a \in \End A \),
\[ (E^{\lambda})^{\wedge (g-1)} \wedge a^*(E^{\lambda q}) = (D^{g-1} \mathbin{.} a^*(D_q)) \cdot v \]
for a suitable generator \( v \) of \( \Hom_{\Z_\ell} (\bigwedge^{2g} T_\ell A, \Z_\ell(g)) \).
Following the proof of \cite{mumford:abelian-varieties} \S 21 Theorem~1 we deduce that
\[ \Tr(a^\dag q a) = \frac{2g}{(D^g)} (D^{g-1} \mathbin{.} a^*(D_q)). \]
Because \( D \) is ample and \( a^*(D_q) \) is effective, this is positive for all \( a \in \End A - \{ 0 \} \).

We shall use \cref{fourth-power-hermitian-form} to prove that there exists \( u \in \rM_4(E) \) such that
\begin{equation} \label{eqn:q-u}
u^\dag \diag_4(q) u = 1.
\end{equation}
Once we have obtained such a \( u \), we can clear denominators by finding an integer~\( n \) such that \( nu \in \rM_4(\End A) \).
Thus in \( \rM_4(\End A) = \End(A^4) \), we have
\[ (nu)^\dag \diag_4(q) \, nu = n^2. \]
We can then carry out the following calculation in \( \Hom(A^4, A^{\vee 4}) \):
\begin{align*}
    n^2 \diag_4(\lambda)
  & = \diag_4(\lambda) \, (nu)^\dag \diag_4(q) \, nu
\\& = (nu)^\vee \diag_4(\lambda) \, \diag_4(q) \, nu
  && \text{(definition of Rosati involution)}
\\& = (nu)^* (\diag_4(\lambda q))
  && \text{(definition of \( (nu)^* \))}
\\& = (nu)^* (\diag_4(f^* \mu))
  && \text{(by \eqref{eqn:def-q})}
\\& = \bigl( \diag_4(f) \circ nu \bigr)^* (\diag_4(\mu)).
\end{align*}
Hence \( \diag_4(f) \circ nu \) is the desired polarized isogeny \( (A, \lambda)^4 \to (B, \mu)^4 \).

To prove that \eqref{eqn:q-u} has a solution,
write \( E \) as a direct product of simple \( \Q \)-algebras.
Because \( \dag \) is a positive involution, it stabilizes each simple factor of \( E \) and restricts to a positive involution of the factor.
Hence it will suffice to solve \eqref{eqn:q-u} independently in each factor and combine the solutions.

We therefore restrict to the case in which \( E \) is simple i.e.\ \( E = \End_D(V) \) for some division algebra \( D \) and some right \( D \)-module~\( V \).
Using \cref{positive-involutions-restrict}, choose a positive involution \( * \) of~\( D \) and a positive definite \( (D, *) \)-Hermitian form \( \psi \colon V \times V \to D \) such that \( \dag \) is the associated adjoint involution.

Let \( \psi_q \colon V \times V \to D \) be defined by
\[ \psi_q(v, w) = \psi(v, qw). \]
By \cref{positive-definite-form-multiply}, this is also a positive definite \( (D, *) \)-Hermitian form.
So by \cref{fourth-power-hermitian-form}, \( \psi_q^{\oplus 4} \) is isometric to \( \psi^{\oplus 4} \), or equivalently, there is a solution to \eqref{eqn:q-u}.

\subsection{Proof of \cref{fourth-power-hermitian-form}}

We are given a division algebra \( D \) over \( \Q \) with a positive involution~\( * \) and two positive definite \( (D, *) \)-Hermitian forms \( \psi_1, \psi_2 \colon V \times V \to D \).
We have to show that \( \psi_1^{\oplus 4} \) and \( \psi_2^{\oplus 4} \) are isometric.

We split the proof into cases depending on the type of \( (D, *) \) in the Albert classification of division algebras with positive involution (see \cite{mumford:abelian-varieties} \S 21 Theorem~2).
In each case we use the classification of \( (D, *) \)-Hermitian forms from chapter~10 of~\cite{scharlau:quadratic-forms}.

Note that \( \psi_1^{\oplus 4} \) and \( \psi_2^{\oplus 4} \) trivially have the same dimension.
They also always have the same signatures at all real places because they are assumed to be positive definite.

\subsubsection*{Type I}
\( D \) is a totally real number field and the involution \( * \) is trivial, so \( (D, *) \)-Hermitian forms are just quadratic forms over \( D \).
Isometry classes of quadratic forms over a number field \( D \) are classified by their dimension, their determinant in \( D^\times/D^{\times 2} \), their Hasse invariant in \( \Br D \) and their signatures at real places of \( D \) (\cite{scharlau:quadratic-forms} Corollary~6.6.6).

The determinant of \( \psi_1^{\oplus 4} \) is the fourth power of \( \det \psi_1 \), so is in \( D^{\times 2} \), and similarly for \( \psi_2^{\oplus 4} \).

It remains to show that the Hasse invariants \( s(\psi_1^{\oplus 4}) \) and \( s(\psi_2^{\oplus 4}) \) are the same.
We shall prove this by proving that \( s(\psi_1^{\oplus 4}) \) is the trivial element of \( \Br D \); the same proof shows that \( s(\psi_2^{\oplus 4}) \) is also trivial.

To prove that \( s(\psi_1^{\oplus 4}) \) is trivial, we will use Lemma~2.12.6 from \cite{scharlau:quadratic-forms}.
This says that
\[ s(\phi \oplus \psi) = s(\phi) s(\psi) \sigma(\det \phi, \det \psi) \]
for any quadratic forms \( \phi \) and \( \psi \) over \( D \), where \( \sigma \colon D^\times/D^{\times 2} \times D^\times/D^{\times 2} \to \Br D \) denotes the Hilbert symbol.
In our case, we deduce that
\[ s(\psi_1^{\oplus 4}) = s(\psi_1^{\oplus 2})^2 \, \sigma(\det \psi_1^{\oplus 2}, \det \psi_1^{\oplus 2}). \]
Since \( s \) takes values in the \( 2 \)-torsion subgroup of \( \Br D \), \( s(\psi_1^{\oplus 2})^2 \) is trivial.
Since \( \det \psi_1^{\oplus 2} \) is a square,
\[ \sigma(\det \psi_1^{\oplus 2}, \det \psi_1^{\oplus 2}) = 1. \]
Hence \( s(\psi_1^{\oplus 4}) \) is trivial.

\subsubsection*{Type II}
\( D \) is a totally indefinite quaternion algebra whose centre is a totally real number field~\( F \) and \( * \) is an orthogonal involution.

There is a localization map on the Witt group of \( (D, *) \)-Hermitian forms
\[ r \colon W(D, *) \to \prod_\fp W(D_\fp, *) \]
where the product on the right hand side runs over all places of \( F \), but this map is not injective.
We will first show that \( [\psi_1^{\oplus 2}] - [\psi_2^{\oplus 2}] \) is in the kernel of \( r \), then use the fact that \( \ker r \) has exponent \( 2 \).

For each non-archimedean place \( \fp \) of \( F \), we first note that by \cite{scharlau:quadratic-forms} Remark~7.6.7 the classification of \( (D_\fp, *) \)-Hermitian forms is equivalent to the classification of \( (D_\fp, \emptybar) \)-skew-Hermitian forms, where \( \emptybar \) denotes the canonical involution of \( D_\fp \).

Hence we can apply \cite{scharlau:quadratic-forms} Theorem~10.3.6: non-singular \( (D_\fp, \emptybar) \)-skew-Hermitian forms are classified by their dimension and their determinant in \( F^\times / F^{\times 2} \).
The determinants of \( \psi_1^{\oplus 2} \) and of \( \psi_2^{\oplus 2} \) are both squares, so \( \psi_1^{\oplus 2} \) and \( \psi_2^{\oplus 2} \) are locally isometric at every non-archimedean place.

At each archimedean place \( \fp \) of~\( F \), \( (D_\fp, *) \cong (\rM_2(\R), \text{transpose}) \) so \( (D_\fp, *) \)-Hermitian forms on \( D_\fp^n \) are just quadratic forms on \( \R^{2n} \).
Hence they are classified by their dimension and signature.
Thus \( \psi_1^{\oplus 2} \) and \( \psi_2^{\oplus 2} \) are locally isometric at archimedean places.

Hence
\[ [\psi_1^{\oplus 2}] - [\psi_2^{\oplus 2}] \in \ker r. \]

According to \cite{lewis:skew-hermitian-forms} Proposition~3,
\[ \ker r \cong (\Z/2\Z)^{s-2} \]
where \( s \) is the number of places of \( F \) at which \( D \) is non-split.
In fact the statement of \cite{lewis:skew-hermitian-forms} Proposition~3 only tells us the order of \( \ker r \), not its precise group structure.
However the group structure can be deduced from the proof of \cite{lewis:skew-hermitian-forms} Proposition~3 or by using the fact that section~4 of \cite{lewis:skew-hermitian-forms} exhibits an explicit homomorphism from \( \ker r \) into a quotient of \( (\Z/2)^s \).

In particular \( \ker r \) has exponent \( 2 \) and so
\[ [\psi_1^{\oplus 4}] - [\psi_2^{\oplus 4}] = 2 \bigl( [\psi_1^{\oplus 2}] - [\psi_2^{\oplus 2}] \bigr) = 0 \]
in \( W(D, *) \).
Since \( \psi_1^{\oplus 4} \) and \( \psi_2^{\oplus 4} \) represent the same element of the Witt group and have the same dimension, they are isometric.

\subsubsection*{Type III}
\( D \) is a totally definite quaternion algebra whose centre is a totally real number field~\( F \) and \( * \) is the canonical involution of \( D \).

According to \cite{scharlau:quadratic-forms} Examples~10.1.8, \( (D, *) \)-Hermitian forms are classified by their dimension and signatures at all real places of \( F \).
As remarked above, this implies that \( \psi_1^{\oplus 4} \) and \( \psi_2^{\oplus 4} \) are isometric.

\subsubsection*{Type IV}
\( D \) is a division algebra whose centre is a CM field \( F \) and \( * \) is an involution of the second kind.
Let \( F_0 \) be the fixed field of \( * \) in \( F \).

By \cite{scharlau:quadratic-forms} Corollary~10.6.6, \( (D, *) \)-Hermitian forms are classified by their dimension, their determinant in \( F_0^\times / \Nm_{F/F_0}(F^\times) \) and their signatures at all real places of \( F_0 \) which do not decompose in \( F \).
In our case \( F \) is a CM field so all real places of \( F_0 \) decompose in \( F \) and the signature condition is empty.

The determinants \( \det(\psi_1^{\oplus 4}) \) and \( \det(\psi_2^{\oplus 4}) \) are squares in \( F_0^\times \) and hence are in \( \Nm_{F/F_0}(F^\times) \) as required.
So \( \psi_1^{\oplus 4} \) and \( \psi_2^{\oplus 4} \) are isometric.

This completes the proof of \cref{fourth-power-hermitian-form}.

\section{Bound for the Degree of Polarized Isogenies} \label{sec:degree-bound}

In this section we prove \cref{polarized-isogeny-bound} and \cref{arith-bound-global}, on the existence of polarized isogenies of polynomially bounded degree.
We will begin with a technical definition of norms on semisimple algebras, then explain how \cref{arith-bound-global} implies \cref{polarized-isogeny-bound} and give an outline of the proof of \cref{arith-bound-global} before we go through all the details of the latter proof.

\subsection{Norms in semisimple algebras}

We define a \defterm{norm} on a semisimple \( \Q \)-algebra \( E \) to be a function \( \Nm_E \colon E \to \Q \) which has the form
\[ \Nm_E(x) = \prod_i \abs{\Nm_{F_i/\Q}(\Nrd_{E_i/F_i}(x_i))}^{\gamma_i} \]
for some positive integers \( \gamma_i \), where \( E = \prod_i E_i \) as a product of simple algebras and \( F_i \) is the centre of \( E_i \).
The \defterm{rank} of the norm is defined to be the integer \( d \) such that
\[ \Nm_E(x) = \abs{x}^d \text{ for } x \in \Q. \]
We say that the norm \( \Nm_E \) is \defterm{\( \dag \)-compatible} if \( \Nm_E \circ \dag = \Nm_E \), where \( \dag \) is an involution of \( E \) (in other words, this requires that \( \gamma_i = \gamma_j \) whenever \( \dag \) exchanges the simple factors \( E_i \) and \( E_j \)).

The purpose of this definition is that ``degree'' is an example of such a norm on the endomorphism algebra of an abelian variety.
In particular we have to allow the exponents \( \gamma_i \) to be greater than \( 1 \) and to depend on \( i \) in order for this to hold for all abelian varieties.

This definition has the following obvious properties:
\begin{enumerate}
\item \( \Nm_E(x) > 0 \) for all \( x \in E^\times \).
\item \( \Nm_E(xy) = \Nm_E(x) \Nm_E(y) \) for all \( x, y \in E \).
\item \( \Nm_E(x) \in \Z \) if \( x \) is an element of an order in \( E \).
\item \( \Nm_E(x) = 1 \) if \( x \) is a unit in an order in \( E \).
\end{enumerate}

\begin{lemma} \label{norm-times-inverse}
Let \( E \) be a semisimple \( \Q \)-algebra, \( R \subset E \) an order and \( \Nm_E \colon E \to \Q \) a norm.

For all \( x \in R - \{ 0 \} \), \( \Nm_E(x) x^{-1} \in R \).
\end{lemma}

\begin{proof}
Define the reduced characteristic polynomial \( P_x(T) \in \Q[T] \) of \( x \in E \) as follows.
Let \( E = \prod_i E_i \) as a product of simple algebras, and let \( F_i \) be the centre of \( E_i \).
Let \( Q_{x,i}(T) \) be the characteristic polynomial over \( \Q \) of \( x_i \) acting on \( E_i \) by left multiplication.
If \( \dim_{F_i} E_i = n_i^2 \), then \( Q_{x,i}(T) = P_{x,i}(T)^{n_i} \) for some polynomial \( P_{x,i}(T) \in \Q[T] \).
We define
\[ P_x(T) = \prod_i P_{x,i}(T). \]

Label the coefficients of \( P_x \) as 
\[ P_x(T) = T^n + a_{n-1} T^{n-1} + \dotsb + a_1 T + a_0. \]
Since \( P_x(x) = 0 \), we get
\begin{equation} \label{eqn:rearranged-char-poly}
-a_0 x^{-1} = x^{n-1} + a_{n-1} x^{n-2} + \dotsb + a_1.
\end{equation}
Since \( x \in R \), the coefficients of \( P_x \) are all in~\( \Z \).
Hence the right hand side of \eqref{eqn:rearranged-char-poly} is in~\( R \).
We deduce that \( a_0 x^{-1} \in R \).

The definition of reduced norms implies that
\[ a_0 = \pm \prod_i \Nm_{F_i/\Q}(\Nrd_{E_i/F_i}(x_i)). \]
Since \( \Nm_{F_i/\Q}(\Nrd_{E_i/F_i}(x_i)) \in \Z \) and \( \gamma_i \geq 1 \) for all \( i \), we deduce that
\[ \Nm_E(x) a_0^{-1} = \pm \prod_i \Nm_{F_i/\Q}(\Nrd_{E_i/F_i}(x_i))^{\gamma_i-1} \]
is an integer.

We conclude that
\[ \Nm_E(x) x^{-1} = (\Nm_E(x) a_0^{-1}) (a_0 x^{-1}) \in R. \qedhere \]
\end{proof}

We also make a local analogue of the above definition of norms.
We define a \defterm{norm} on a semisimple \( \Q_p \)-algebra \( E_p \) to be a function \( \Nm_{E_p} \colon E_p \to \Q \) of the form
\[ \Nm_{E_p}(x) = \prod_i \abs{\Nm_{F_i/\Q_p}(\Nrd_{E_i/F_i}(x_i))}_p^{-\gamma_i} \]
for some positive integers \( \gamma_i \), where \( E = \prod_i E_i \) as a product of simple algebras and \( F_i \) is the centre of \( E_i \).
The \defterm{rank} of \( \Nm_{E_p} \) is defined to be the positive integer \( d \) such that
\[ \Nm_{E_p}(x) = \abs{x}_p^{-d} \text{ for all } x \in \Q_p. \]

Note that the exponents in the definition of a local norm are negative.
This is because \( \abs{x}_p \leq 1 \) when \( x \) is a \( p \)-adic integer, and so local norms \( \Nm_{E_p} \) satisfy property~(3) above.
Indeed, it is simple to check that all of properties (1)--(4) above and \cref{norm-times-inverse} hold for a local norm \( \Nm_{E_p} \).

Furthermore, if \( \Nm_E \) is a norm on a semisimple \( \Q \)-algebra \( E \), then the extensions of \( \Nm_E \) to the localizations \( E \otimes_\Q \Q_p \) satisfy
\[ \Nm_E(x) = \prod_p \Nm_{E \otimes_\Q \Q_p}(x) \text{ for all } x \in E. \]

\subsection{Proof that \cref{arith-bound-global} implies \cref{polarized-isogeny-bound}}

We are given principally polarized abelian varieties \( (A, \lambda) \) and \( (B, \mu) \) and isogenies \( f, g \colon A \to B \) such that \( f \) is a polarized isogeny and \( \deg g = n \).
We want to prove the existence of a polarized isogeny \( h \colon A \to B \) of degree at most \( cn^k \), where \( c \) and \( k \) depend only on~\( (A, \lambda) \).

We will apply \cref{arith-bound-global} to \( R = \End A \) and \( E = R \otimes_\Z \Q \), with \( \dag \) being the Rosati involution with respect to the polarization \( \lambda \).
The norm is given by \( \Nm_E(a) = \deg a \) for \( a \in \End A \) (this is defined to be \( 0 \) if \( a \) is not an isogeny), extended homogeneously to \( E \) i.e.
\[ \Nm_E(a) = \deg(na)/n^{2\dim A} \text{ where } n \text{ is a non-zero integer such that } na \in \End A. \]
By \cite{milne:abelian-varieties-old} Proposition~12.12, this is a norm on \( E \) as defined above, with degree \( 2 \dim A \).

Set
\[ a = g^{-1} f \in \End A \otimes_\Z \Q. \]

Let \( q \) be an element of \( \End A \) such that \( g^* \mu = \lambda \circ q \).
A calculation shows that
\begin{align*}
    \lambda a^\dag qa
	 = a^\vee \lambda qa
	 = a^*(\lambda q)
	 = a^*(g^* \mu)
	 = (ga)^*(\mu)
	 = f^* \mu.
\end{align*}
Hence the fact that \( f \) is a polarized isogeny implies that
\[ a^\dag qa \in \Q^\times. \]
We can therefore apply \cref{arith-bound-global} to obtain \( b \in \End A \) such that
\[ b^\dag q b \in \Z - \{ 0 \} \text{ and } \Nm_E(b) \leq c\Nm_E(q)^{d-1/2}. \]

The fact that \( b^\dag q b \in \Z - \{ 0 \} \) implies that \( h = g \circ b \) is a polarized isogeny \( A \to B \).

The definition of \( q \) implies that
\[ \Nm_E(q) = (\deg g)^2 = n^2 \]
and so the bound from \cref{arith-bound-global} gives
\[ \deg h = n \Nm_E(b) \leq c n \Nm_E(q)^{d-1/2} = c n (n^2)^{d-1/2} = cn^{2d} \]
where \( d = 2\dim A \).

\subsection{Outline of the proof of \cref{arith-bound-global}} \label{ssec:outline}

Before we come to the proof of \cref{arith-bound-global} in general, we will first look at the case where \( E \) is a number field and \( R \) is its ring of integers.
We will sketch a proof that in this case, there is some \( b \in R \) satisfying \( b^\dag qb \in \Z - \{ 0 \} \) and whose norm is bounded by some polynomial in \( \Nm_E(q) \), but we will not seek to optimize the bound.
Indeed this sketch will give a weaker exponent than is stated in \cref{arith-bound-global}.

We begin by looking for an ideal instead of an element of \( R \) which satisfies the conclusion of \cref{arith-bound-global}.
In other words, we look for an ideal \( \fb \subset R \) which has suitably bounded norm and which satisfies
\begin{equation} \label{eqn:ideals}
\fb^\dag q \fb = mR \text{ for some } m \in \Z.
\end{equation}

Take \( a \) as in the hypothesis of \cref{arith-bound-global}.
Multiplying it by a rational integer, we may assume without loss of generality that \( a \in R \).
Then the principal ideal \( aR \) satisfies \eqref{eqn:ideals}, showing that the set of ideals \( \fb \) which satisfy \eqref{eqn:ideals} is non-empty.

In order to find a solution to \eqref{eqn:ideals} with small norm, we work locally in \( R_p = R \otimes_\Z \Z_p \) for each rational prime \( p \), looking at ideals \( \fb_p \subset R_p \) which satisfy
\begin{equation} \label{eqn:ideals-local}
\fb_p^\dag q \fb_p = mR_p \text{ for some } m \in \Z_p.
\end{equation}
If \( qR \) is coprime to \( pR \), then clearly \( \fb_p = R_p \) satisfies \eqref{eqn:ideals-local}.
So we only need to consider the finitely many primes \( p \) for which \( pR \) is not coprime to \( qR \).

Denote the factorizations into prime ideals in \( R_p \) of \( pR_p \) and of \( qR_p \) by
\[ pR_p = \fp_1^{e_1} \dotsm \fp_r^{e_r}
\text{ and }
qR_p = \fp_1^{k_1} \dotsm \fp_r^{k_r}. \]
Given an ideal \( \fb_p \subset R_p \) satisfying \eqref{eqn:ideals-local}, let its prime factorization be
\[ \fb_p = \fp_1^{\beta_1} \dotsm \fp_r^{\beta_r}. \]

If \( p \) divides \( \fb_p \), then we can replace \( \fb_p \) by \( p^{-1} \fb_p \) and it will still be an ideal of \( R_p \) satisfying \eqref{eqn:ideals-local}.
Hence we can assume that \( p \) does not divide \( \fb_p \).
This implies that \( \beta_i < e_i \) for some \( i \).

Suppose that \( \fb_p^\dag q \fb_p = p^t R_p \).
Then
\[ t = (2\beta_i + k_i)/e_i \text{ for all } i. \]
Applying this for an \( i \) at which \( \beta_i < e_i \), we get
\[ t < 2 + \max(k_1, \dotsc, k_r) \leq 2 + v_p(\Nm(qR_p)). \]
A calculation then shows that
\[ \Nm(\fb_p) < p^d \Nm(qR_p)^{(d-1)/2} \]
where \( d = [E:\Q] \).
Since we are assuming that \( qR \) is not coprime to \( pR \), \( \Nm(qR_p) \geq p \) and so this implies
\[ \Nm(\fb_p) < \Nm(qR_p)^{(3d-1)/2}. \]

Letting \( \fb \) be the product of the ideals \( \fb_p \cap R \), we get an ideal of \( R \) which satisfies \eqref{eqn:ideals} and such that
\[ \Nm(\fb) \leq \abs{\Nm_E(q)}^{(3d-1)/2}. \]

Using finiteness of the class group, we may replace the ideal \( \fb \) by a principal ideal at the cost of a constant factor in the norm bound.
Thus there are \( b \in R \), \( u \in R^\times \) and \( m \in \Z \) such that
\[ b^\dag q b = um. \]
Using the fact that \( R^\times \) is finitely generated, we can remove the unit~\( u \) at the cost of another constant factor.

\medskip

The argument sketched above relies on \( E \) being a field.
When \( E \) is not a field, our proof will have the same local-global structure, but we will work with adèles instead of ideals.
We will begin by proving a local version of \cref{arith-bound-global} for all primes \( p \), namely \cref{arith-bound-local-non-split}.
However, this local version, with a constant \( c_p \) for each prime \( p \), is not sufficient to deduce a global result: we need to know that the constants \( c_p \) are \( 1 \) for almost all~\( p \).
This is given by \cref{arith-bound-local-split}.
Once we have these two local results, we will then use the adelic version of finiteness of the class group to obtain \cref{arith-bound-global}.

The local results \cref{arith-bound-local-non-split,arith-bound-local-split} are results about lattices and Hermitian forms over division algebras over local fields.
(In the commutative case sketched above, the relevant hermitian forms are on \( 1 \)-dimensional vector spaces, and so can simply be described by scalars.)

In the case of \cref{arith-bound-local-non-split}, our proof does not use hermitian forms explicitly.
However it uses the \( p \)-adic polar decomposition of Benoist and Oh~\cite{benoist-oh:polar}, which can be seen as a generalization of the diagonalization of quadratic forms over a field.

In the proof of \cref{arith-bound-local-split}, we work directly with Hermitian forms.
This corollary applies only to primes at which \( E_p \) is split, so we only need to consider Hermitian forms over a field instead of over a division algebra.
We get the necessary integrality ingredients by using properties of maximal lattices.

\subsection{Local calculations -- non-split case}

We will prove the local version of \cref{arith-bound-global}, valid for all primes \( p \) but with a non-trivial constant \( c_p \) for every prime \( p \).
The exponent \( (d-1)/2 \) in this local result (\cref{arith-bound-local-non-split}) is better than the \( d-1/2 \) of \cref{arith-bound-global} but this is not important -- the weaker exponent in \cref{arith-bound-global} comes from \cref{arith-bound-local-split}.
Our primary ingredient is the \( p \)-adic polar decomposition of Benoist and Oh (\cite{benoist-oh:polar} Theorem~1.1), applied to the element \( a \) such that \( a^\dag qa \in \Q_p^\times \).

Here is an outline of the proof of \cref{arith-bound-local-non-split}.
We rearrange the hypothesis \( a^\dag qa \in \Q_p^\times \) to obtain \( q^{-1} \in \Q_p^\times aa^\dag \).
The \( p \)-adic polar decomposition allows us to write \( a = ksh \) where \( s \) is fixed by \( \dag \) and is in one of a fixed finite collection of commutative subalgebras of \( E_p \).
We can easily deal with \( k \) and \( h \), so that instead of \( aa^\dag \) we have to look at \( s^2 \).
(In the case where \( E_p \) is a matrix algebra over a field and \( \dag \) is the transpose involution, replacing \( aa^\dag \) by \( s^2 \) corresponds to diagonalizing the quadratic form with matrix \( aa^\dag \).)

Some calculations allow us to construct a \( \Q_p \)-multiple \( us \) of \( s \) such that \( (us)^2 \) is in the order \( R_p \) and has bounded norm.
The fact that \( s \) is in one of a fixed set of commutative subalgebras of \( E \) allows us to deduce that a bounded multiple of \( us \) is in \( R_p \).
Reversing the calculations finishes the proof.

We will need the following lemma once we know that \( us \) is in a fixed commutative subalgebra and that \( (us)^2 \) is in \( R_p \).
The key point in the proof of \cref{square-root-integer-in-torus} is that the (unique) maximal order in a commutative algebra is integrally closed.

\begin{lemma} \label{square-root-integer-in-torus}
Let \( E_p \) be a semisimple \( \Q_p \)-algebra and \( R_p \subset E_p \) an order.
Let \( L \subset E_p \) be a commutative \( \Q_p \)-subalgebra.

There is a positive rational integer \( c \) depending on \( E_p \), \( R_p \) and \( L \) such that: for all \( x \in L \), if \( x^2 \in R_p \) then \( cx \in R_p \).
\end{lemma}

\begin{proof}
Let \( \fo_L \) denote the maximal order in \( L \).

\( R_p \cap L \) is a \( \Z_p \)-subalgebra of \( L \) which is finitely generated as a \( \Z_p \)-module, so it is contained in \( \fo_L \).
\( R_p \cap L \) is also open in \( L \) (because \( R_p \) is open in \( E_p \)), so it has finite index in \( \fo_L \).
Hence there is \( c \in \N \) such that
\[ c\fo_L \subset R_p \cap L. \]

Now if \( x \in L \) and \( x^2 \in R_p \), then \( x^2 \in \fo_L \).
Since \( \fo_L \) is integrally closed, it follows that \( x \in \fo_L \) and so \( cx \in R_p \cap L \).
\end{proof}

Let us recall the \( p \)-adic polar decomposition of Benoist and Oh.
Note that we use a different definition of involution of a group from \cite{benoist-oh:polar}: for us, an involution reverses the order of multiplication, while in \cite{benoist-oh:polar} an involution preserves the order of multiplication.
Hence \( \dag \colon G \to G \) is an involution in our sense if and only if \( g \mapsto (g^\dag)^{-1} \) is an involution in the sense of \cite{benoist-oh:polar}.
This leads to the cosmetic differences between the definitions of \( \gH \) and of \( (\Q_p, \dag) \)-split tori given below and those in \cite{benoist-oh:polar}.

Let \( \gG \) be a connected reductive algebraic group over \( \Q_p \), let \( \dag \) be an involution of \( \gG \), and let \( \gH \) be the algebraic subgroup
\[ \gH = \{ h \in \gG \mid hh^\dag = 1 \}. \]
We say that a torus \( \gS \subset \gG \) is \defterm{\( (\Q_p, \dag) \)-split} if \( \gS \) is split over \( \Q_p \) and \( s^\dag = s \) for all \( s \in S \).
By a theorem of Helminck and Wang~\cite{helminck-wang} there are finitely many \( \gH(\Q_p) \)-conjugacy classes of maximal \( (\Q_p, \dag) \)-split tori in \( \gG \).
Choose representatives~\( \gS_i \) for these \( \gH(\Q_p) \)-conjugacy classes of maximal \( (\Q_p, \dag) \)-split tori.

Theorem~1.1 of~\cite{benoist-oh:polar} asserts that there exists a compact subset \( K \subset \gG(\Q_p) \) such that
\begin{equation} \label{eqn:polar-decomposition-group}
\gG(\Q_p) = K \left( \bigcup_i \gS_i(\Q_p) \right) \gH(\Q_p).
\end{equation}

\begin{lemma} \label{arith-bound-local-non-split}
Let \( (E_p, \dag) \) be a semisimple \( \Q_p \)-algebra with involution, \( R_p \subset E_p \) a \( \dag \)-stable order, and \( \Nm_{E_p} \) a \( \dag \)-compatible norm on \( E_p \) of rank~\( d \).

There exists a constant \( c_p \) depending only on \( (R_p, \dag, \Nm_{E_p}) \) such that:

For every \( q \in R_p \), if there exists \( a \in E_p \) such that \( a^\dag q a \in \Q_p^\times \),
then there exists \( b \in R_p \) such that
\[ b^\dag q b \in \Z_p - \{ 0 \} \text{ and } \Nm_{E_p}(b) \leq c_p\Nm_{E_p}(q)^{(d-1)/2}. \]
\end{lemma}

\begin{proof}
Let \( \gG \) be the reductive \( \Q_p \)-algebraic group with functor of points
\[ \gG(A) = (E_p \otimes_{\Q_p} A)^\times. \]
The involution \( \dag \) of \( E_p \) induces an involution of \( \gG \).
Let \( \gH \) be the subgroup of \( \dag \)-unitary elements and let \( \gS_i \) be representatives for the \( \gH(\Q_p) \)-conjugacy classes of maximal \( (\Q_p, \dag) \)-split tori in \( \gG \), as above.

Choose a compact subset \( K \subset \gG(\Q_p) \) satisfying \eqref{eqn:polar-decomposition-group}.
Because \( K \) is compact, its elements have bounded denominators.
Hence after replacing \( K \) by a scalar multiple, we may assume that \( K \subset R_p \).
Since \( K^{-1} \) is also compact, we can choose a constant \( \cp{1} \in \N \) such that \( \cp{1} K^{-1} \subset R_p \).

Let \( m = a^\dag qa \).
Because \( m \) is invertible and in the centre of \( E_p \), we can rearrange this to get
\begin{equation} \label{eqn:m-1qaa}
q^{-1} = m^{-1} \, aa^\dag.
\end{equation}

Let \( n = \Nm_{E_p}(q) \).
By \cref{norm-times-inverse}, \( nq^{-1} \in R_p \).
So \eqref{eqn:m-1qaa} implies that
\[ nm^{-1} \, aa^\dag \in R_p. \]

Using the \( p \)-adic polar decomposition, write
\[ a = ksh \]
with \( k \in K \), \( s \in \bigcup_i \gS_i(\Q_p) \) and \( h \in \gH(\Q_p) \).
Substituting this in the previous equation, and using the facts that \( hh^\dag = 1 \) and \( s = s^\dag \), we get that
\[ nm^{-1} \, ks^2k^\dag \in R_p. \]
Multiplying by \( \cp{1} k^{-1} \) on the left and \( \cp{1} (k^{-1})^\dag \) on the right, we get that
\[ \cp{1}^2 \, nm^{-1} \, s^2 \in R_p. \]

Choose \( e \in \Z_p \) such that \( \cp{1}^2 nm^{-1}e \) is a square in \( \Q_p^\times \) and \( v_p(e) = 0 \) or~\( 1 \).
Let \( u \) denote a square root in \( \Q_p^\times \) of \( \cp{1}^2 nm^{-1}e \).
We get that
\[ u^2 s^2 \in R_p. \]

Furthermore, \( us \in \bigcup_i \gS_i(\Q_p) \) because the scalars \( \Q_p^\times \) are contained in every maximal \( (\Q_p, \dag) \)-split torus of \( \gG \).
For each \( i \), the subalgebra \( L_i \subset E_p \) generated by \( \gS_i(\Q_p) \) is commutative.
We can therefore apply \cref{square-root-integer-in-torus} inside each \( L_i \).
We deduce that there is a constant \( \cp{2} \) (depending on \( R_p \) and the \( \gS_i \)) such that
\[ \cp{2} us \in R_p. \]

Letting
\[ b = \cp{2}u \, ks, \]
we get that \( b \in R_p \) and
\[ b b^\dag = \cp{2}^2 u^2 \, ks^2k^\dag = \cp{1}^2 \cp{2}^2 \, nm^{-1}e \, aa^\dag = \cp{1}^2 \cp{2}^2 \, ne \, q^{-1} \]
where the last equality follows from \eqref{eqn:m-1qaa}.
Since \( \cp{1}^2 \cp{2}^2 \, ne \) is in the centre of \( E_p \), we can rearrange this to obtain
\[ b^\dag q b = \cp{1}^2 \cp{2}^2 \, ne \in \Z_p - \{ 0 \}. \]

Finally we bound the norm of \( b \).
The above equation gives us that
\[ n \Nm_{E_p}(b)^2  = \Nm_{E_p}(\cp{1}^2 \cp{2}^2 \, ne) = \cp{1}^{2d} \cp{2}^{2d} \, n^d \Nm_{E_p}(e). \]
Since \( v_p(e) = 0 \) or \( 1 \), \( \Nm_{E_p}(e) \leq p^d \).
Hence
\[ \Nm_{E_p}(b)^2  \leq  \cp{1}^{2d} \cp{2}^{2d} \, p^d n^{d-1}. \]
So the lemma is proved with constant \( c_p = (\cp{1}^{2d} \cp{2}^{2d} p^d)^{1/2} \).
\end{proof}

\subsection{Local calculations -- split case}

Our goal now is to prove that for all but finitely many primes \( p \), \cref{arith-bound-local-non-split} holds  with \( c_p = 1 \).
Specifically we will prove this for all \( p \) such that \( E_p \) is split (meaning that \( E_p \) is a product of matrix algebras over fields), the centre of \( E_p \) is a product of unramified extensions of \( \Q_p \), \( R_p \) is a maximal order in \( E_p \) and \( p \neq 2 \).

The first step in this proof applies to simple algebras with involution (\cref{partial-bound-local-split}).
We will then obtain a result for a semisimple algebra with involution (\cref{arith-bound-local-split}) by applying \cref{partial-bound-local-split} to each of its simple factors.
In order to do this, it is not enough just to show that, for each simple factor, there exists some \( b \) such that \( b^\dag q b \in \Z_p - \{ 0 \} \).
In order that the solutions for different simple factors combine together, it is necessary that the scalars \( b^\dag qb \) should be the same in each factor.
Therefore we state \cref{partial-bound-local-split} in a form which allows to choose the \( m' \in \Z_p - \{ 0 \} \) for which we want to solve \( b^\dag qb = m' \), subject to certain constraints.

The statement of \cref{partial-bound-local-split} does not mention norms at all.
The norm bound comes from the choice of \( m' \), which will be made in the proof of \cref{arith-bound-local-split}.
The naïve choice for \( m' \) leads to a bound
\[ \Nm_{E_p}(b) \leq p^{d/2} \Nm_{E_p}(q)^{(d-1)/2}. \]
If \( q \not\in R_p^\times \), then \( \Nm_{E_p}(q) \geq p \) and so we can remove the power of \( p \) from the above bound, at the cost of weakening the exponent.
On the other hand if \( q \in R_p^\times \), then \( \Nm_{E_p}(q) = 1 \) and we have to have \( m' \in \Z_p^\times \) in order to get the desired bound from \cref{partial-bound-local-split}.
In order to achieve this, we will need an additional result on the classification of unimodular Hermitian forms (\cref{unit-forms-local-split}).

\begin{lemma} \label{partial-bound-local-split}
Let \( (E_p, \dag) \) be a simple \( \Q_p \)-algebra with involution and \( R_p \subset E_p \) a \( \dag \)-stable order.
Suppose that \( E_p \) is split, its centre is an unramified extension of \( \Q_p \), \( R_p \) is a maximal order in \( E_p \) and \( p \neq 2 \).

Let \( q \in R_p \) and \( a \in E_p \) be such that \( a^\dag q a \in \Q_p^\times \).

Let \( m = a^\dag q a \) and let \( m' \in \Z_p - \{ 0 \} \) be such that
\( m'q^{-1} \in R_p \)
and \( m'm^{-1} \) is a square in \( \Q_p^\times \).

Then there exists \( b \in R_p \) such that
\[ b^\dag q b = m'. \]
\end{lemma}

\begin{proof}
Let \( F \) be the centre of \( E_p \) and \( F_0 \) the subfield of \( F \) fixed by \( \dag \).
Since \( E_p \) is split, it is isomorphic to \( \End_F(V) \) for some \( F \)-module \( V \).
By \cref{adjoint-involutions}, there is an \( F_0 \)-bilinear form \( \psi \colon V \times V \to F \) such that \( \dag \) is the adjoint involution with respect to \( \psi \), and we are in the setting of section~\ref{ssec:lattices}.

Since \( R_p \) is a maximal order in \( E_p \), we can find a lattice \( \Lambda \subset V \) whose stabilizer is~\( R_p \).
Since \( R_p \) is \( \dag \)-stable, \cref{symmetric-stabilizer-implies-maximal-lattice} implies that \( \Lambda \) is a maximal lattice with respect to \( \psi \).
Let \( \fa \) be the scale of \( \Lambda \) with respect to \( \psi \).

Let \( \psi_q \) be the \( F_0 \)-bilinear form \( V \times V \to F \) given by
\[ \psi_q(v, w) = \psi(v, qw). \]
A calculation shows that
\begin{equation} \label{eqn:neq-1Lambda}
   \{ v \in V \mid \psi_q(v, m'q^{-1} \Lambda) \subset m' \fa \}
 = \{ v \in V \mid \psi(v, \Lambda) \subset \fa \}
 = \Lambda
\end{equation}
where the second equality holds because \( \Lambda \) is \( \fa \)-maximal (see the proof of \cref{symmetric-stabilizer-implies-maximal-lattice}).
Since \( m'q^{-1} \in R_p \), we have
\( m'q^{-1} \Lambda \subset \Lambda \)
and hence \eqref{eqn:neq-1Lambda} implies that
\[ \psi_q(m'q^{-1} \Lambda, m'q^{-1} \Lambda) \subset m' \fa. \]

Hence by \cref{lattice-in-maximal}, there exists a lattice \( \Lambda' \) which is \( m' \fa \)-maximal with respect to~\( \psi_q \) and which contains \( m'q^{-1} \Lambda \).
Note that \( \Lambda' \) must be contained in
\[ \{ v \in V \mid \psi_q(v, m'q^{-1} \Lambda) \subset m' \fa \} \]
and so by \eqref{eqn:neq-1Lambda}, \( \Lambda' \subset \Lambda \).

The fact that \( m = a^\dag qa \) implies that \( a\Lambda \) is an \( m\fa \)-maximal lattice with respect to \( \psi_q \).
Choosing a square root \( u \in \Q_p^\times \) of \( m^{-1} m' \), we deduce that \( ua\Lambda \) is an \( m'\fa \)-maximal lattice with respect to \( \psi_q \).

Hence by \cref{maximal-lattices-are-isometric}, \( (ua \Lambda, \psi_q) \) is isometric to \( (\Lambda', \psi_q) \).
It follows that there exists \( b \in E_p \) such that
\[ \Lambda' = b \Lambda \text{ and } b^\dag q b = (ua)^\dag q ua. \]
The fact that \( \Lambda' \subset \Lambda \) implies that \( b \in R_p \), while a calculation gives
\[ b^\dag q b = (ua)^\dag q ua = u^2 a^\dag qa = u^2 m = m'. \qedhere \]
\end{proof}

When \( q \in R_p^\times \), we will use the following lemma to enable us to choose an \( m' \in \Z_p^\times \) for use in \cref{partial-bound-local-split}.
In case (i) we can choose \( m' \in \Z_p^\times \) such that \( m' (a^\dag qa)^{-1} \) is a square.
In case (ii) we will let \( m' = 1 \) and apply \cref{partial-bound-local-split} to \( b \) coming from \cref{unit-forms-local-split} instead of the original~\( a \).

\begin{lemma} \label{unit-forms-local-split}
Let \( (E_p, \dag) \) be a simple \( \Q_p \)-algebra with involution and \( R_p \subset E_p \) a \( \dag \)-stable order.
Suppose that \( E_p \) is split, its centre is an unramified extension of \( \Q_p \), \( R_p \) is a maximal order in \( E_p \) and \( p \neq 2 \).

Let \( q \in R_p \) and \( a \in E_p \) be such that \( a^\dag q a \in \Q_p^\times \).
Suppose further that \( q \in R_p^\times \).

Then either:
\begin{enumerate}[(i)]
\item \( v_p(a^\dag qa) \) is even; or
\item there exists \( b \in E_p^\times \) such that \( b^\dag q b = 1 \).
\end{enumerate}
\end{lemma}

\begin{proof}
We use the same notation \( F \), \( F_0 \), \( V \), \( \psi \), \( \psi_q \) as in the proof of \cref{partial-bound-local-split}.

The hypothesis that \( a^\dag qa \in \Q_p^\times \) tells us that \( (V, \psi_q) \) is isometric to \( (V, m\psi) \) for the scalar \( m = a^\dag qa \in \Q_p^\times \).
In order to show that there exists \( b \in E_p^\times \) such that \( b^\dag q b = 1 \), we have to show that in fact \( (V, \psi_q) \) is isometric to \( (V, \psi) \).
There is one case in which it is not true that \( (V, \psi_q) \) is isometric to \( (V, \psi) \); in this case we will show instead that \( v_p(m) \) is even, leading to the two cases in the conclusion of the lemma.

We proceed in cases according to the type of the form \( \psi \).

\begin{enumerate}
\setlength{\itemsep}{1em}

\item \( F = F_0 \) and \( \psi \) is symmetric.

Choose a basis for \( V \) as an \( F \)-vector space, inducing an isomorphism \( E_p \cong \rM_n(F) \).
Since \( R_p \) is a maximal order in \( E_p \), we can choose the basis such that \( R_p \) corresponds to \( \rM_n(\fo_F) \).

We can write the quadratic form \( \psi \) in the form
\[ \psi(v, w) = v^t zw \]
for some symmetric matrix \( z \in \rM_n(F) \).
Then the adjoint involution is given by
\[ x^\dag = z^{-1} x^t z. \]
Since \( R_p \) is \( \dag \)-stable, this implies that \( z \) normalizes \( R_p \).
The normalizer of \( R_p \) is \( F^\times R_p^\times \), so after multiplying \( \psi \) by a scalar in \( F^\times \) (which does not change the defining property of \( \psi \), namely that its adjoint involution is~\( \dag \)), we may assume that \( z \in R_p^\times = \GL_n(\fo_F) \) and so \( \psi \) is a unimodular quadratic form.

Since \( q \in R_p^\times \), this implies that \( zq \in R_p^\times \) and hence \( \psi_q \) is also unimodular.
Since \( p \neq 2 \), \cite{omeara:quadratic-forms} 92:1 implies that unimodular quadratic forms over \( \fo_F \) are classified by their dimension and their determinant in \( F^\times / F^{\times 2} \).

The fact that \( (V, \psi_q) \) and \( (V, m\psi) \) are isometric implies that
\begin{equation} \label{eqn:square-classes2}
\det(\psi_q) = \det(m\psi) = m^n\det(\psi) \text{ in } F^{\times} / F^{\times 2}.
\end{equation}

If \( n \) is even, then \eqref{eqn:square-classes2} tells us at once that 
\( \det(\psi_q) = \det(\psi) \) in \( F^\times / F^{\times 2} \)
and so \( (V, \psi_q) \) is isometric to \( (V, \psi) \).

If \( n \) is odd, then \eqref{eqn:square-classes2} implies that
\[ \det(\psi_q) = m\det(\psi) \text{ in } F^\times / F^{\times 2}. \]
Since \( \psi_q \) and \( \psi \) are both unimodular, we deduce that
\[ m \in \fo_F^\times \, F^{\times 2}. \]
Since \( F/\Q_p \) is unramified, this implies that \( v_p(m) \) is even.

\item \( F = F_0 \) and \( \psi \) is skew-symmetric.

Isometry classes of skew-symmetric forms over a field are classified by their dimension alone, so \( (V, \psi) \) and \( (V, \psi_q) \) are isometric.

\item \( F/F_0 \) is a quadratic extension of fields and \( \psi \) is Hermitian.

For local fields \( F \) and \( F_0 \), isometry classes of non-singular \( F/F_0 \)-Hermitian forms are classified by their dimension and their determinant in \( F_0^\times / \Nm_{F/F_0}(F^\times) \) (see \cite{scharlau:quadratic-forms} Examples~10.1.6(ii)).

Since \( \psi \) and \( \psi_q \) are both unimodular, \( \det \psi \) and \( \det \psi_q \) are both in \( \fo_{F_0}^\times \).
Since the extension \( F/F_0 \) is unramified, \( \Nm_{F/F_0}(F^\times) \) contains \( \fo_{F_0}^\times \).
Hence \( (V, \psi) \) and \( (V, \psi_q) \) are isometric.

\item \( F = F_0 \times F_0 \) and \( \psi \) is Hermitian.

By \cite{scharlau:quadratic-forms} Example~7.2.7, all non-singular \( (F_0 \times F_0)/F_0 \)-Hermitian forms of the same dimension are isometric, in particular \( (V, \psi) \) and \( (V, \psi_q) \).
\qedhere
\end{enumerate}
\end{proof}

\begin{corollary} \label{arith-bound-local-split}
Let \( (E_p, \dag) \) be a semisimple \( \Q_p \)-algebra with involution, \( R_p \subset E_p \) a \( \dag \)-stable order, and \( \Nm_{E_p} \) a \( \dag \)-compatible norm on \( E_p \) of rank~\( d \).
Suppose that \( E_p \) is split, its centre is a product of unramified extensions of \( \Q_p \), \( R_p \) is a maximal order in \( E_p \) and \( p \neq 2 \).

For every \( q \in R_p \), if there exists \( a \in E_p \) such that \( a^\dag q a \in \Q_p^\times \),
then there exists \( b \in R_p \) such that
\[ b^\dag q b \in \Z_p - \{ 0 \} \text{ and } \Nm_{E_p}(b) \leq \Nm_{E_p}(q)^{d-1/2}. \]
\end{corollary}

\begin{proof}
Write the algebra \( E_p \) as a direct product
\[ E_p = \prod_i E_i \]
where each \( E_i \) is a simple algebra with involution.
Since \( R_p \) is a maximal order in~\( E_p \), it is a direct product of maximal orders \( R_i \subset E_i \).

Let \( m = a^\dag qa \) and \( n = \Nm_{E_p}(q) \).

We have two cases, depending on whether \( q \in R_p^\times \) or not.

\subsubsection*{Case 1}
If \( q \in R_p^\times \), then we apply \cref{unit-forms-local-split} to each factor \( E_i \).
If conclusion (i) of \cref{unit-forms-local-split} holds for at least one factor \( E_i \), then we know that \( v_p(m) \) is even.
We can therefore choose some \( m' \in \Z_p^\times \) such that \( m'm^{-1} \) is a square in \( \Q_p^\times \).
Since \( q \in R_p^\times \), \( m'q^{-1} \in R_p \) and so we can apply \cref{partial-bound-local-split} in each factor \( E_i \) to \( q_i \), \( a_i \) and \( m' \) to obtain \( b_i \in R_i \) such that \( b_i^\dag q_i b_i = m' \).

Otherwise (still within the case \( q \in R_p^\times \)), conclusion (ii) of \cref{unit-forms-local-split} holds in every factor~\( E_i \).
In other words, for each \( i \), there exists \( b_i \in E_i^\times \) such that \( b_i^\dag q_i b_i = 1 \).
In this case, simply let \( m' = 1 \).
\Cref{partial-bound-local-split} allows us to upgrade \( b_i \in E_i \) to \( b_i \in R_i \).

Letting \( b \) be the element of \( R_p \) whose components are the \( b_i \), we get that
\[ b q b^\dag = m' \in \Z_p - \{ 0 \}. \]
The fact that \( m' \in \Z_p^\times \) implies that \( b \in R_p^\times \) and so
\[ \Nm_{E_p}(b) = 1 = \Nm_{E_p}(q)^{d-1/2}. \]

\subsubsection*{Case 2}
If \( q \not\in R_p^\times \), choose \( e \in \Z_p \) such that \( v_p(e) = 0 \) or \( 1 \) and \( nem^{-1} \) is a square in \( \Q_p^\times \).
By \cref{norm-times-inverse}, \( neq^{-1} \in R_p \).
So in each factor \( E_i \), we can apply \cref{partial-bound-local-split} to \( q_i \) and~\( a_i \) with \( m' = ne \).

The resulting \( b_i \) fit together to give \( b \in R_p \) such that
\[ b^\dag q b = ne \in \Z_p - \{ 0 \}. \]
This implies that
\[ \Nm_{E_p}(b)^2 \Nm_{E_p}(q) = \Nm_{E_p}(ne) \leq n^d p^d \]
and hence
\[ \Nm_{E_p}(b) \leq n^{(d-1)/2} p^{d/2}. \]
Since \( q \not\in R_p^\times \), we have \( p \leq n \) so this implies that \( \Nm_{E_p}(b) \leq n^{d-1/2} \).
\end{proof}

\subsection{Global arguments}

We are now ready to complete the proof of \cref{arith-bound-global}, deducing it from \cref{arith-bound-local-non-split,arith-bound-local-split}.

We are given a semisimple \( \Q \)-algebra \( E \) with involution~\( \dag \) and a \( \dag \)-stable order \( R \subset E \).
We will work in the adelic points of the group \( \gU \) of \( \dag \)-quasi-unitary elements of \( R \), that is, the \( \Z \)-group scheme with functor of points
\[ \gU(A) = \{ u \in (R \otimes_\Z A)^\times \mid uu^\dag \in A^\times \}. \]

We are given \( q \in R \) and \( a \in E \) such that \( a^\dag qa \in \Q^\times \).
Our first step is to choose \( b_p \in R_p \) for each \( p \) such that \( b_p^\dag qb_p \in \Q_p^\times \), and the norms of \( b_p \) are bounded according to \cref{arith-bound-local-non-split,arith-bound-local-split}.
We then look at
\[ u_p = b_p^{-1} a. \]
Some calculations show that the \( u_p \) are components of an adelic element \( \mathbf{u} \in \gU(\A_f) \).

Finiteness of the adelic class set of \( \gU \) allows us to write \( \mathbf{u} \) as a product of an element \( x \) of \( \gU(\Q) \), an element \( \mathbf{g}_i \) from a fixed finite set, and an element \( \mathbf{y} \) of \( \gU(\A_f) \) which is a unit at every prime.
Then
\[ b = ax^{-1} \]
is an element of \( E \) which satisfies \( b^\dag qb \in \Z - \{ 0 \} \).
Another calculation shows that the norm of \( b \) is not far away from \( \prod_p \Nm_{E_p}(b_p) \), and hence satisfies the required bound.

\begin{proof}[Proof of \cref{arith-bound-global}]
By hypothesis, we have \( q \in R \) and \( a \in E \) such that
\[ a^\dag q a \in \Q^\times. \]
Let \( m = a^\dag q a \).
Since \( m \) is in the centre of \( E \), we can deduce that \( m q^{-1} = a a^\dag \).

Applying \cref{arith-bound-local-non-split} in each localization \( R_p = R \otimes_\Z \Z_p \),
we get \( b_p \in R_p \) such that
\[ b_p^\dag q b_p \in \Z_p - \{ 0 \} \text{ and } \Nm_{E_p}(b_p) \leq c_p\Nm_{E_p}(q)^{d-1/2}. \]

For all but finitely many primes \( p \), we can use \cref{arith-bound-local-split} to choose \( b_p \) as above with \( c_p = 1 \); furthermore the set of exceptional \( p \) depends only on \( (R, \dag) \).

For all but finitely many \( p \), we have \( \Nm_{E_p}(q) = 1 \) (this time, the set of exceptions depends on \( q \)).
If also \( c_p = 1 \), then the above bound says that \( \Nm_{E_p}(b_p) = 1 \).
For these \( p \), \cref{norm-times-inverse} tells us that \( b_p \in R_p^\times \).

\medskip

Let
\[ u_p = b_p^{-1} a \in E_p^\times. \]
Then
\[ u_p u_p^\dag = b_p^{-1} a a^\dag b_p^{\dag -1} = b_p^{-1} mq^{-1} b_p^{\dag -1} = m (b_p^\dag q b_p)^{-1} \in \Q_p^\times \]
so \( u_p \in \gU(\Q_p) \).
Furthermore, for all but finitely many \( p \), \( a \in R_p^\times \) and \( b_p \in R_p^\times \) so the \( u_p \) are components of an adelic element \( \mathbf{u} \in \gU(\A_f) \).

By \cite{platonov-rapinchuk} Theorem~5.1, the double coset space
\[ \prod \gU(\Z_p) \backslash \gU(\A_f) / \gU(\Q) \]
is finite.
Choose representatives \( \mathbf{g}_1, \dotsc, \mathbf{g}_r \) for these double cosets.
By multiplying them by suitable elements of \( \Q^\times \), we clear denominators so that
\[ \mathbf{g}_i \in \prod_p R \otimes_\Z \Z_p \]
for each~\( i \).

We can decompose \( \mathbf{u} \) as
\[ \mathbf{u} = \mathbf{y} \mathbf{g}_i x \]
for some \( \mathbf{y} \in \prod_p \gU(\Z_p) \), \( \mathbf{g}_i \) among our chosen double coset representatives and \( x \in \gU(\Q) \).

\medskip

Let \( b = ax^{-1} \in E \).
We claim that \( b \) satisfies the conditions of the proposition.

At each prime \( p \), we have that
\[ b = ax^{-1} = b_p u_p x^{-1} = b_p y_p g_{i,p} \in R_p. \]
Hence \( b \in \bigcap_p R_p = R \).

Next
\[ b^\dag q b = x^{\dag -1} a^\dag qa x^{-1} = m x^{\dag -1} x^{-1}. \]
Now \( x^{\dag -1} x^{-1} \in \Q^\times \) because \( x \in \gU(\Q) \), so we conclude that
\[ b^\dag q b \in \Q^\times. \]
In fact \( b^\dag q b \in \Z - \{ 0 \} \) because \( b \) and \( q \) are both in \( R \) and \( R \cap \Q^\times = \Z - \{ 0 \} \).

Finally we have to bound
\[ \Nm_E(b) = \prod_p \bigl( \Nm_{E_p}(b_p) \Nm_{E_p}(y_p) \Nm_{E_p}(g_{i,p}) \bigr). \]

For all \( p \), \( \Nm_{E_p}(y_p) = 1 \) because \( y_p \in R_p^\times \).

For each \( i \), \( \Nm_{E_p}(g_{i,p}) = 1 \) for all but finitely many \( p \), and so the following constant is well-defined:
\[ c_0 = \max_{1 \leq i \leq r} \prod_p \Nm_{E_p}(g_{i,p}). \]
Recall that we can choose \( \mathbf{g}_i \) depending only on \( (R, \dag) \) so \( c_0 \) depends only on \( (R, \dag, \Nm_E) \).

Using the bounds on \( \Nm_{E_p}(b_p) \) from \cref{arith-bound-local-non-split,arith-bound-local-split}, we get that
\[ \Nm_{E}(b) \leq c_0 \prod_p \bigl( c_p \Nm_{E_p}(q)^{d/1-2} \bigr) = c_0 \bigl( \prod_p c_p \bigr) \Nm_E(q)^{d-1/2} \]
and so \cref{arith-bound-global} holds with \( c = c_0 \prod_p c_p \).
\end{proof}

\subsection*{Acknowledgements}

I am grateful to Emmanuel Ullmo and Andrei Yafaev for useful conversations about the work in this paper.
I would like to thank the MathOverflow user WKC who suggested the method of proof of \cref{partial-bound-local-split}.
I also thank the referee for helpful comments.

During the writing of this paper, the author was supported by a doctoral grant from Université Paris Sud and by European Research Council grant 307364 ``Some problems in Geometry of Shimura Varieties.''

\bibliographystyle{amsalpha}
\bibliography{polarised}

\providecommand{\bysame}{\leavevmode\hbox to3em{\hrulefill}\thinspace}
\providecommand{\MR}{\relax\ifhmode\unskip\space\fi MR }
\providecommand{\MRhref}[2]{%
  \href{http://www.ams.org/mathscinet-getitem?mr=#1}{#2}
}
\providecommand{\href}[2]{#2}
\begin{thebibliography}{KMRT98}

\bibitem[BO07]{benoist-oh:polar}
Y.~Benoist and H.~Oh, \emph{Polar decomposition for {$p$}-adic symmetric
  spaces}, Int. Math. Res. Not. IMRN (2007), no.~24, Art. ID rnm121, 20.

\bibitem[HW93]{helminck-wang}
A.~G. Helminck and S.~P. Wang, \emph{On rationality properties of involutions
  of reductive groups}, Adv. Math. \textbf{99} (1993), no.~1, 26--96.

\bibitem[KMRT98]{knus+:involutions}
M.-A. Knus, A.~Merkurjev, M.~Rost, and J.-P. Tignol, \emph{The book of
  involutions}, American Mathematical Society Colloquium Publications, vol.~44,
  American Mathematical Society, Providence, RI, 1998.

\bibitem[Kot92]{kottwitz:positive-involutions}
R.~E. Kottwitz, \emph{Points on some {S}himura varieties over finite fields},
  J. Amer. Math. Soc. \textbf{5} (1992), no.~2, 373--444.

\bibitem[Lew82]{lewis:skew-hermitian-forms}
D.~W. Lewis, \emph{Quaternionic skew-{H}ermitian forms over a number field}, J.
  Algebra \textbf{74} (1982), no.~1, 232--240.

\bibitem[Mil86]{milne:abelian-varieties-old}
J.~S. Milne, \emph{Abelian varieties}, Arithmetic geometry ({S}torrs, {C}onn.,
  1984), Springer, New York, 1986, pp.~103--150.

\bibitem[Mum70]{mumford:abelian-varieties}
D.~Mumford, \emph{Abelian varieties}, Tata Institute of Fundamental Research
  Studies in Mathematics, No. 5, Oxford University Press, London, 1970.

\bibitem[MW93]{mw:isogeny-avs}
D.~Masser and G.~W{\"u}stholz, \emph{Isogeny estimates for abelian varieties,
  and finiteness theorems}, Ann. of Math. (2) \textbf{137} (1993), no.~3,
  459--472.

\bibitem[O'M63]{omeara:quadratic-forms}
O.~T. O'Meara, \emph{Introduction to quadratic forms}, Die Grundlehren der
  mathematischen Wissenschaften, Bd. 117, Academic Press Inc., Publishers, New
  York, 1963.

\bibitem[Orr15]{orr:andre-pink}
M.~Orr, \emph{Families of abelian varieties with many isogenous fibres}, J.
  Reine Angew. Math. \textbf{705} (2015), 211--231.

\bibitem[Pin05]{pink:conj}
R.~Pink, \emph{A combination of the conjectures of {M}ordell-{L}ang and
  {A}ndr\'e-{O}ort}, Geometric methods in algebra and number theory, Progr.
  Math., vol. 235, Birkh\"auser Boston, Boston, MA, 2005, pp.~251--282.

\bibitem[PR94]{platonov-rapinchuk}
V.~Platonov and A.~Rapinchuk, \emph{Algebraic groups and number theory}, Pure
  and Applied Mathematics, vol. 139, Academic Press Inc., Boston, MA, 1994.

\bibitem[Sch85]{scharlau:quadratic-forms}
W.~Scharlau, \emph{Quadratic and {H}ermitian forms}, Grundlehren der
  Mathematischen Wissenschaften, vol. 270, Springer-Verlag, Berlin, 1985.

\bibitem[Shi63]{shimura:alternating-forms}
G.~Shimura, \emph{Arithmetic of alternating forms and quaternion hermitian
  forms}, J. Math. Soc. Japan \textbf{15} (1963), 33--65.

\bibitem[Shi97]{shimura:hermitian-forms-book}
\bysame, \emph{Euler products and {E}isenstein series}, CBMS Regional
  Conference Series in Mathematics, vol.~93, Washington, DC, 1997.

\end{thebibliography}

\end{document}